\documentclass{paper}

\usepackage{amsmath}
\usepackage{amsthm}
\usepackage{amsfonts}
\usepackage{amssymb}
\usepackage[hidelinks]{hyperref}

\usepackage{dsfont}
\usepackage{multicol}

\usepackage{float}
\usepackage{graphicx}
\usepackage{sidecap}
\usepackage{caption}
\usepackage{subcaption}
\usepackage{framed}

\theoremstyle{plain}
\newtheorem*{theorem*}{Theorem}
\newtheorem{theorem}{Theorem}[section]

\newtheorem{claim}{Claim}

\newtheorem{corollary}[theorem]{Corollary}

\newtheorem{lemma}[theorem]{Lemma}

\newtheorem{proposition}[theorem]{Proposition}

\theoremstyle{definition}
\newtheorem{remark}{Remark}

\newcommand{\R}{\mathds{R}}
\newcommand{\sn}[1]{\mathds{S}^{#1}}

\newcommand{\forma}[1]{\langle #1 \rangle}

\newcommand{\grad}{\text{grad}}
\newcommand{\graph}{\text{graph}}
\newcommand{\norma}[1]{\left\Vert #1 \right\Vert}

\newcommand{\n}{\nabla}

\newcommand{\rar}{\R^2 \rtimes_A \R}
\newcommand{\raz}{\R^2 \rtimes_A \{0\}}

\newcommand{\xt}{\widetilde{x}}

\newcommand{\yt}{\widetilde{y}}

\newcommand{\tr}{\text{trace}}

\newcommand{\Sbn}{\overline{\Sigma}_n}

\newcommand{\Md}[4]{\left(\begin{array}{cc}{#1}&{#2}\\{#3}&{#4}\end{array}\right)}
\newcommand{\abs}[1]{\vert #1 \vert}

\newcommand{\len}{\text{length}}

\newcommand{\PSL}{\widetilde{PSL}}

\newcommand{\Omegab}{\overline{\Omega}}

\title{The mean curvature equation on semidirect products $\rar$: Height estimates and Scherk-like graphs}

\author{\'Alvaro Kr\"uger Ramos\thanks{Supported by CNPq - Brazil}}

\begin{document}
\maketitle

\begin{abstract}
On the ambient space of a Lie group with a left invariant metric that is isometric and isomorphic to a semidirect product $\rar$, we consider a domain $\Omega\subseteq \raz$ and vertical $\pi$-graphs over $\Omega$ and study the partial differential equation a function $u:\Omega \rightarrow \R$ must satisfy in order to have prescribed mean curvature $H$. Using techniques from quasilinear elliptic equations we prove that if a $\pi-$graph has non-negative mean curvature, then it satisfy some uniform height estimates that depend on $\Omega$ and on a parameter $\alpha$, given a priori. When $\tr(A) > 0$, these estimates imply that the oscillation of a minimal graph assuming the same constant value $n$ along the boundary tends to zero when $n\rightarrow + \infty$ and goes to $+ \infty$ if $n\rightarrow - \infty$. Furthermore, we use some of the estimates, allied with techniques from Killing graphs, to prove the existence of minimal $\pi-$graphs assuming the value $0$ along a piecewise smooth curve $\gamma$ with endpoints $p_1,\,p_2$ and having as boundary $\gamma \cup (\{p_1\}\times[0,\,+\infty))\cup(\{p_2\}\times[0,\,+\infty))$.

\vspace{.15cm}

\vspace{.1cm}

\end{abstract}

\section{Introduction}\label{secIntroSemid}

The subject of minimal surfaces is among the most beautiful - and also most studied - objects in Differential Geometry, and many times minimal graphs play a important role on the field. Many deep results may be obtained from looking to a minimal surface as a local graph and bringing techniques from partial differential equations to help solve geometric questions.

On a series of papers (\cite{MMPR,MMPR2,MMPR3,MP}), W. Meeks III, P. Mira, J. P\'erez and A. Ros studied constant mean curvature spheres on three-dimensional homogeneous spaces and semidirect products of the form $\rar$ came to play a very important role on their proofs. It is true that any simply connected metric Lie group of dimension 3 is either $SU(2)$ or $\widetilde{\text{PSL}}(2,\R)$ endowed with a left invariant metric or it is isomorphic and isometric to a semidirect product $\rar$ with its canonical left invariant metric, where $A \in M_2(\R)$ is some $2\times2$ square matrix (see the work of W. Meeks and J. P\'erez \cite{MP}).

Although a lot of work has been done on the last years on the ambient space of semidirect products (or, more generally, on homogeneous spaces of dimension 3), the theory of minimal $\pi-$graphs on semidirect products is still on development, and many interesting questions on this field remain open. For instance is it true that a minimal $\pi-$graph on $\rar$ is stable? Or, more generally, if the left invariant Gauss map of a surface $\Sigma \subseteq \rar$ is contained on a hemisphere, does it imply $\Sigma$ is stable? Another related question that unfortunately remains unsolved concerns the uniqueness of minimal $\pi-$graphs with prescribed boundary.

There are two main difficulties when dealing with minimal $\pi-$graphs on semidirect products $\rar$: the first one is that vertical translations $(x,y,z)\mapsto (x,y,z+t)$ are not isometries of the ambient space. In particular this affects the mean curvature operator so its high order terms coefficients depend on the solution, and the comparison principle (for instance Theorem 10.1 of \cite{GT} and its generalizations) does not apply. The second one is that, unless $\tr(A) = 0$, constant functions do not provide minimal graphs, so there is no maximum principle.

On this paper, we consider a convex domain $\Omega \subseteq \raz$ with piecewise smooth boundary and exhibit the differential equation some function $u:\Omega \rightarrow \R$ must satisfy for its $\pi-$graph $\text{graph}(u) = \{(x,\,y,\,u(x,y)) \in \rar;\,(x,\,y,\,0) \in \Omega\}$ to have prescribed mean curvature function. Depending on the trace and on the determinant of $A$ such PDE is going to have a different behaviour. For instance, when $\tr(A) = 0$, if $u:\Omega \rightarrow \R$ is a function whose $\pi-$graph has non negative mean curvature $H \geq 0$ with respect to the upwards orientation, then it satisfies the maximum principle

\begin{equation}\label{maxprin}
\displaystyle \sup_{\partial\Omega} u = \sup_{\Omega} u.
\end{equation}

This property was first observed by W. Meeks, P. Mira, J. P\'erez and A. Ros (\cite{MMPR3}, stated on its generality as Lemma~\ref{tracelemma} below), and we remark that \eqref{maxprin} does not hold on the case of $\tr(A) > 0$, even for $H\equiv 0$: a minimal graph that is constant along its boundary necessarily assumes an interior maximum and it is not constant, as horizontal planes (representing constant functions) are no longer minimal. It becomes a natural question to ask if there is a \emph{maximal jump} these minimal graphs that are constant along the boundary can attain, and this question is answered via height estimates of partial differential equations.

On Section \ref{secHE}, we prove some of our main results. Given $\Omega \subseteq \raz$ and a parameter $\alpha \in \R$, we obtain a constant $C = C(\text{diam}(\Omega),\,\alpha)$ such that every function $u:\Omega \rightarrow \R$ with mean curvature (not necessarily constant) $H \geq 0$ satisfy 

\begin{equation}\label{maineq}
\displaystyle \sup_{\Omega} u \leq \max\left\{\sup_{\partial \Omega}u,\,\alpha\right\} + C.
\end{equation}

\noindent This is Theorem \ref{main}. Still on Section \ref{secHE} it is then proved that the dependence of the a priori constant $\alpha$ on \eqref{maineq} is essential (Theorem \ref{thmabs}) for the validity of the result. We also use the freedom on the parameter $\alpha$ in order to obtain that (on the case $\tr(A) > 0$, see Theorem~\ref{oscilation} for details) the oscillation of a family of solutions to the problem

\begin{equation*}
\left\{ \begin{array}{l}
\Sigma = \graph(u) \text{ is a minimal surface of } \rar\\
u\vert_{\partial \Omega} = c \in \R
\end{array}
\right.
\end{equation*}

\noindent converges to zero when $c$ approaches $+ \infty$ and goes to infinite, if $c \rightarrow - \infty$.

We finish the paper on Section~\ref{secScherk}, where we bring techniques from Killing graphs, in addition to the estimates on the coefficients of the mean curvature operator obtained on Section~\ref{secHE}, to generalize an argument of A. Menezes \cite{Men} to any semidirect product $\rar$, and obtain the existence of minimal $\pi-$graphs which are similar to the fundamental piece of the doubly periodic Scherk surface of $\R^3$, on Theorem \ref{thmScherks}.

\vspace{0.1cm}

\noindent \textbf{Acknowledgements:} I would like to thank prof. Joaqu\'in P\'erez for his most valuable advice and hospitality during my stay at Granada, which led to the present paper, and to prof. Jaime Ripoll for his continuous support, specially for his help on the last section of this work, that is part of my doctoral thesis.

\section{Semidirect products $\rar$}\label{rar}

This section is to give a brief review on semidirect products $\rar$. We follow the notation and construction of W. Meeks and J. P\'erez, \cite{MP}.

Let $H,\,V$ to be two groups and let $\varphi:V \rightarrow \text{Aut}(H)$ a group homomorphism between $V$ and the group of automorphisms of $H$. Then, the \emph{semidirect product between $H$ and $V$ with respect to $\varphi$}, denoted by $G = H\rtimes_\varphi V$, is the cartesian product $H \times V$ endowed with the group operation $*: G \times G \rightarrow G$ given by 

\begin{equation*}
(h_1,\,v_1)*(h_2,\,v_2) = (h_1\cdot\varphi_{v_1}(h_2),\,v_1v_2).
\end{equation*}

With this group operation, then both $H$ and $V$ can be viewed as subgroups of $G$ and $H \triangleleft G$ is identified to a normal subgroup of $G$. This construction comes to generalize the notion of direct product of groups, where the operation on the cartesian product $H \times V$ would be the product operation $(h_1,\,v_1)*(h_2,\,v_2) = (h_1h_2,\,v_1v_2)$. It is clear that this notion can be recovered from the one of a semidirect product, when considering the automorphism $\varphi(v) = Id_H$ being the constant map into the identity of the group $H$.

Even on the particular case of $H = \R^2$ and $V = \R$ being two abelian groups, it is possible to obtain a great variety of groups via the semidirect product of $\R^2$ and $\R$, depending uniquely on the choice of the (now 1-parameter) family of automorphisms of $\R^2$. Precisely, with the exceptions of $SU(2)$ (not diffeomorphic to $\R^3$) and $\PSL(2,\,\R)$ (has no normal subgroup of dimension 2), it is possible to construct all three dimensional simply connected Lie groups (see \cite{MP}) using the following setting: fix a matrix $A \in M_2(\R)$,

\begin{equation}\label{A}
A = \Md{a}{b}{c}{d}
\end{equation}

\noindent and, for each $z \in \R$ we consider the automorphism of $\R^2$ generated by the exponential map of $Az$, $e^{Az}:\R^2\rightarrow\R^2$. Then, we let 

\begin{equation*}
\begin{array}{cccc}
\varphi:&\R &\rightarrow &\text{Aut}(\R^2)\\
&z & \mapsto & e^{Az},
\end{array}
\end{equation*}

\noindent and we define $\rar = \R^2\rtimes_{\varphi}\R$ as the semidirect product between $\R^2$ and $\R$ with respect to the automorphisms generated by $e^{Az}$. Explicitly, the semidirect product $\rar$ is the set $\R^3 = \R^2 \times \R$ endowed with the group operation $*$ given by

\begin{equation}\label{oprar1}
(x_1,\,y_1,\,z_1) * (x_2,\,y_2,\,z_2) = \left(\left(\begin{array}{c} x_1 \\y_1\end{array}\right) + e^{Az}\left(\begin{array}{c} x_2 \\y_2\end{array}\right),\,z_1+z_2\right).
\end{equation}

Using the notation of \cite{MP}, we denote the exponential map $e^{Az}$ by

\begin{equation}\label{eAz}
e^{Az} = \Md{a_{11}(z)}{a_{12}(z)}{a_{21}(z)}{a_{22}(z)},
\end{equation}

\noindent and observe that the vector fields defined by

\begin{equation}\label{Ei}
E_1(x,y,z) = a_{11}(z)\partial_x+a_{21}(z)\partial_y,\ E_2(x,y,z) = a_{12}(z)\partial_x+a_{22}(z)\partial_y,\ E_3 = \partial_z
\end{equation}

\noindent are left invariant and extend the canonical basis $\{\partial_x(0),\,\partial_y(0),\,\partial_z(0)\}$ at the origin of $\R^3$. Moreover, if we let

\begin{equation}\label{Fi}
F_1 = \partial_x,\ F_2 = \partial_y,\ F_3(x,y,z) = (ax+by)\partial_x+(cx+dy)\partial_y+\partial_z,
\end{equation}

\noindent it follows that each $F_i$ is a right invariant vector field of $\rar$, in particular they are Killing fields with respect to any left invariant metric on $\rar$.

The metric to be considered on $\rar$ is the \emph{canonical left invariant metric} (cf. \cite{MP}), that is the one given by stating that $\{E_1,\,E_2,\,E_3\}$ are unitary and orthogonal to each other everywhere. In particular, as it holds

\begin{eqnarray*}
\partial_x(x,\,y,\,z) & = & a_{11}(-z)E_1+a_{21}(-z)E_2\\
\partial_y(x,\,y,\,z) & = & a_{12}(-z)E_1+a_{22}(-z)E_2,
\end{eqnarray*}

\noindent we can express the metric of $\rar$ in coordinates as

\begin{eqnarray*}
ds^2 & = & \left[a_{11}(-z)^2+a_{21}(-z)^2\right]dx^2+\left[a_{12}(-z)^2+a_{22}(-z)^2\right]dy^2+dz^2\\
&& + \left[a_{11}(-z)a_{12}(-z)+a_{21}(-z)a_{22}(-z)\vphantom{a_{11}(x)^2}\right](dx\otimes dy + dy\otimes dx).
\end{eqnarray*}

Now, we remark that $e^{-Az} = \left(e^{Az}\right)^{-1}$, so, as $\det(e^{Az}) = e^{z\tr(A)}$, we have

\begin{equation*}
\Md{a_{11}(-z)}{a_{12}(-z)}{a_{21}(-z)}{a_{22}(-z)}  =  e^{-z\tr(A)}\Md{a_{22}(z)}{-a_{12}(z)}{-a_{21}(z)}{a_{11}(z)},
\end{equation*}

\noindent and we can introduce the notation

\begin{eqnarray}
Q_{11}(z)\ = \ \forma{\partial_x,\,\partial_x} & = & e^{-2z\tr(A)}\left[a_{21}(z)^2+a_{22}(z)^2\right]\nonumber\\
Q_{22}(z)\ = \ \forma{\partial_y,\,\partial_y} & = & e^{-2z\tr(A)}\left[a_{11}(z)^2+a_{12}(z)^2\right]\label{Qij}\\
Q_{12}(z)\ = \ \forma{\partial_x,\,\partial_y} & = & -e^{-2z\tr(A)}\left[a_{11}(z)a_{21}(z)+a_{12}(z)a_{22}(z)\right]\nonumber
\end{eqnarray}

\noindent to obtain that the metric $ds^2$ is expressed by

\begin{equation}\label{ds2}
ds^2 = Q_{11}(z)dx^2+Q_{22}(z)dy^2+dz^2+Q_{12}(z)(dx\otimes dy + dy\otimes dx).
\end{equation}

We notice that, if $A,\,B \in M_2(\R)$ are two congruent matrices, on the sense that there is some orthogonal matrix $P \in O(2)$ such that $B = PAP^{-1}$, then the groups $\rar$ and $\R^2\rtimes_{B}\R$, endowed with their respective canonical left invariant metrics are \emph{isomorphic and isometric}, and the map that makes the identification is a simple rotation on horizontal planes induced by $P$:

\begin{equation}\label{isoiso}
\begin{array}{rccl}
\varphi:&\rar&\rightarrow & \R^2\rtimes_B \R\\
&(x,y,z)&\mapsto & (P(x,y),z).
\end{array}
\end{equation}

\noindent We also remark that the Lie brackets of $\rar$ satisfy

\begin{equation}\label{Liebracket}
[E_1,\,E_2] = 0, \ [E_3,\,E_1] = aE_1 + cE_2, \ [E_3,\,E_2] = bE_1 + dE_2,
\end{equation}

\noindent so Levi-Civita equation implies that the Riemannian connection of $\rar$ is given by

\vspace{0.2cm}
\begin{tabular}{rcl|rcl|rcl}
\hspace{-0.1cm}$\n_{E_1}E_1$ &$\hspace{-0.1cm} = \hspace{-0.1cm}$ &$ aE_3$ &$\n_{E_1}E_2$ &$ \hspace{-0.1cm}=\hspace{-0.1cm} $&$\frac{b+c}{2}E_3 $&$\n_{E_1}E_3$ & $\hspace{-0.1cm}=\hspace{-0.1cm} $&$-aE_1-\frac{b+c}{2}E_2$ \\
\hspace{-0.1cm}$\n_{E_2}E_1$ &$ \hspace{-0.1cm}=\hspace{-0.1cm}$ &$ \frac{b+c}{2}E_3 $ &$\n_{E_2}E_2$ &$ \hspace{-0.1cm}=\hspace{-0.1cm}$ &$ dE_3 $&$\n_{E_2}E_3 $&$ = $&$ -\frac{b+c}{2}E_1-dE_2$\\
\hspace{-0.1cm}$\n_{E_3}E_1$ &$\hspace{-0.1cm} =\hspace{-0.1cm}$ &$ \frac{c-b}{2}E_2 $ &$\n_{E_3}E_2$ &$ \hspace{-0.1cm}=\hspace{-0.1cm} $&$ \frac{b-c}{2}E_1$ &$\n_{E_3}E_3 $&$ = $&$0$. 
\end{tabular}

\vspace{0.2cm}

It is important to notice two properties of planes on $\rar$: first, we observe that the metric $ds^2$ is invariant by rotations by angle $\pi$ around the vertical lines $\{(x_0,\,y_0,\,z);\ z \in \R\}$, so vertical planes are minimal surfaces of $\rar$. Moreover, horizontal planes $\{z =c\}$ have $E_3$ as an unitary normal vector field, so they have constant mean curvature (with respect to the upward orientation) given by $H = \tr(A)/2$. In particular, horizontal planes of $\rar$ are going to be minimal if and only if $\tr(A) = 0$.

\begin{figure}
\centering
\begin{subfigure}[b]{0.48\textwidth}
\includegraphics[width=\textwidth]{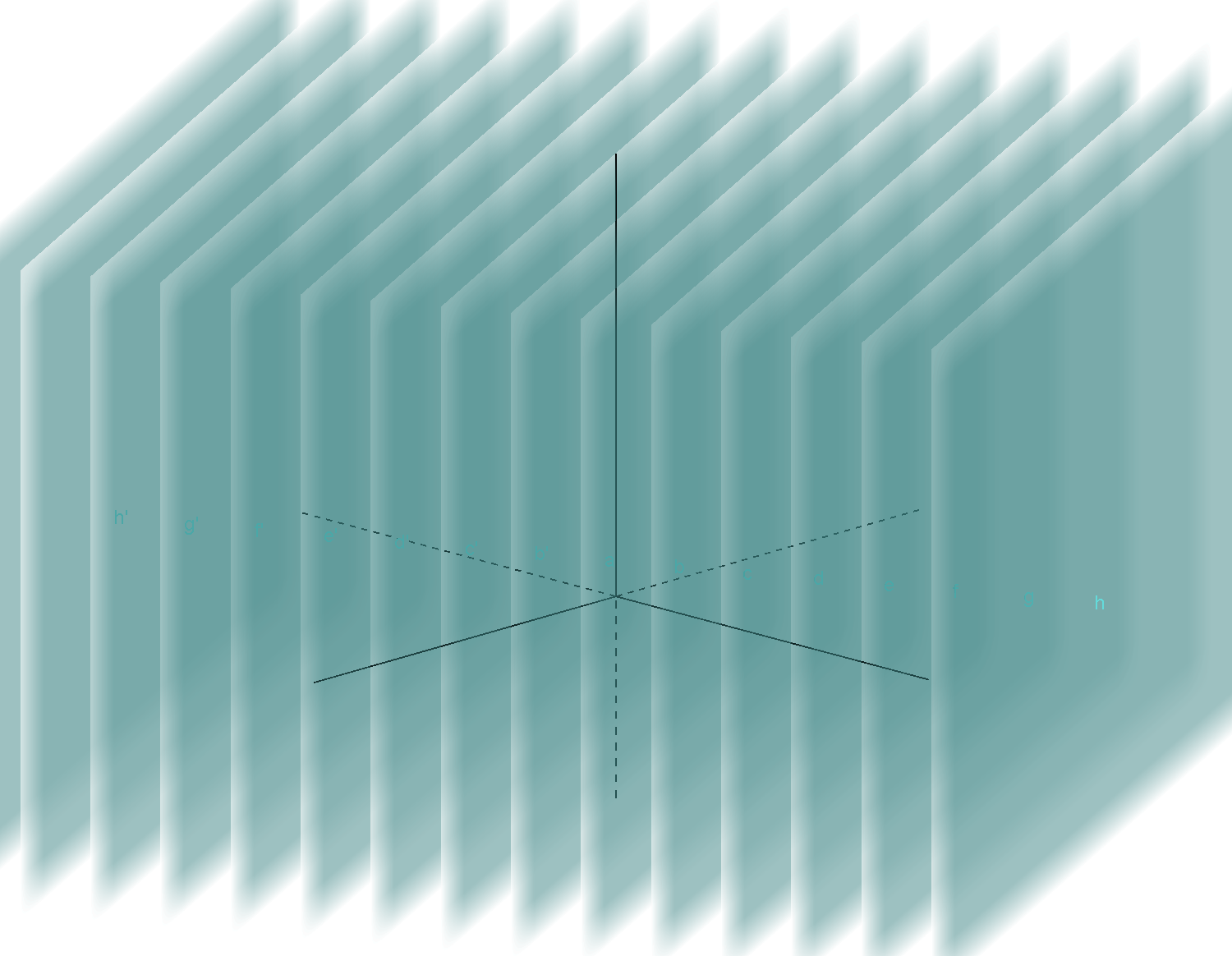}
\caption{A foliation of $\rar$ by vertical (minimal) planes}
\end{subfigure}
\begin{subfigure}[b]{0.48\textwidth}
\includegraphics[width=\textwidth]{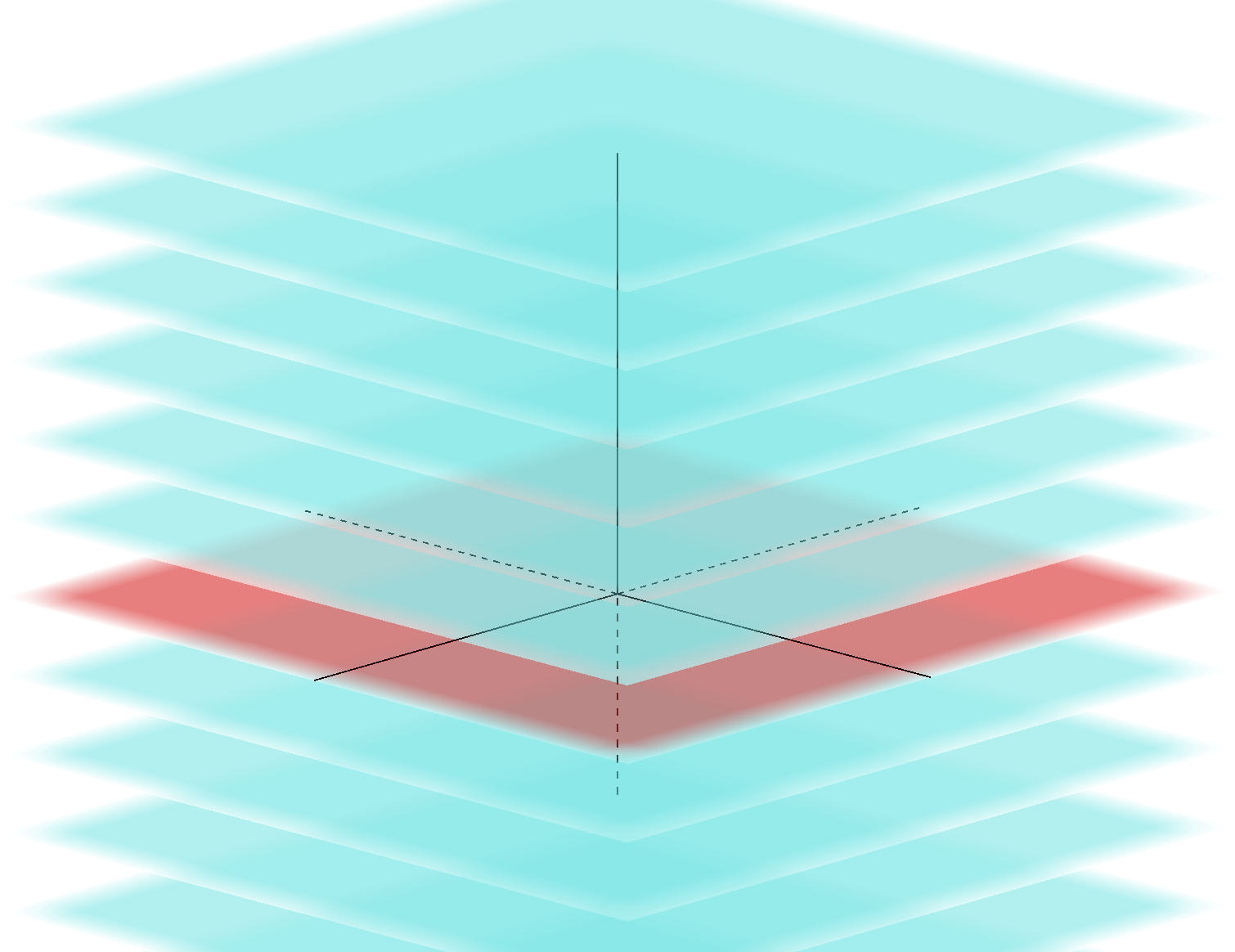}
\caption{The foliation of $\rar$ by horizontal (CMC) planes}
\end{subfigure}
\caption{\small{On semidirect products $\rar$ every \emph{vertical} plane is a minimal surface. Horizontal planes are flat, have constant mean curvature $H = \tr(A)/2$ and the subgroup $\mathbb{H} = \raz$ (highlighted on the above right picture) is normal on $\rar$.}}
\end{figure}

However, the difference between the cases $\tr(A) = 0$ and $\tr(A) \neq 0$ go further than horizontal planes being minimal: concerning the classification of simply connected Lie groups of dimension 3, we notice that W. Meeks and J. P\'erez, \cite{MP} proved that any \emph{non unimodular}\footnote{A group $G$ is said to be unimodular if $\text{det}\left(\text{Ad}_g\right) = 1$ for all $g \in G$} Lie group of dimension 3 is isomorphic and isometric to a semidirect product $\rar$, endowed with its left invariant metric, where $A \in M_2(\R)$ is such that $\tr(A) \neq 0$ (Lemma 2.11, \cite{MP}). Moreover, they also prove that, with the exceptions of $SU(2)$ and $\PSL(2,\,\R)$, all other unimodular metric Lie groups are isomorphic and isometric to a semidirect product $\rar$, with $\tr(A) = 0$ (Section 2.6 and Theorem 2.15, \cite{MP}). Herein, we reefer to the cases $\tr(A) = 0$ or $\tr(A) \neq 0$ respectively as the \emph{unimodular} and \emph{non unimodular} case.

\section{Mean curvature equation and height estimates}\label{secHE}

On this section, we consider a smooth open domain $\Omega \subseteq \raz$ and a function $u:\Omega \rightarrow \R$. We define the $\pi$-graph of $u$ by 

$$\Sigma = \text{graph}(u) = \{(x,\,y,\,u(x,y))\in \rar;\,(x,\,y,\,0) \in \Omega\}.$$

When oriented with respect to the upwards direction, the mean curvature of $\Sigma$ is given by

\begin{eqnarray}
H\hspace{-0.3cm} & = &\hspace{-0.3cm}\left.\frac{e^{2u\tr(A)}}{2W^3}\right[ u_{xx}\left(Q_{22}(u)\hspace{-0.08cm}+\hspace{-0.08cm}u_y^2\right)+u_{yy}\left(Q_{11}(u)\hspace{-0.08cm}+\hspace{-0.08cm}u_x^2\right) - 2u_{xy}\left(Q_{12}(u)\hspace{-0.08cm}+\hspace{-0.08cm}u_xu_y\right)\nonumber\\
&&\left.\hspace{-0.4cm}\phantom{\frac{e^{2u\tr(A)}}{2W^3}}+G_1(u)u_x^2+G_2(u)u_y^2+G_3(u)u_xu_y+(a+d)e^{-2u\tr(A)}\right]\hspace{-0.1cm},\label{propPDE}
\end{eqnarray}

\noindent where $Q_{ij}$ are the coefficients of the metric of $\rar$, defined as on \eqref{Qij}, $G_i:\R\rightarrow \R$ are the functions given by

\begin{eqnarray}
G_1(z)\hspace{-0.3cm} & = & \hspace{-0.3cm}e^{-2z\tr(A)}\big((2a+d)a_{11}(z)^2+(a+2d)a_{12}(z)^2+(b+c)a_{11}(z)a_{12}(z)\big)\nonumber\\
G_2(z)\hspace{-0.3cm} & = &\hspace{-0.3cm} e^{-2z\tr(A)}\big((2a+d)a_{21}(z)^2+(a+2d)a_{22}(z)^2+(b+c)a_{21}(z)a_{22}(z)\big)\nonumber\\
G_3(z)\hspace{-0.3cm} & = &\hspace{-0.3cm} e^{-2z\tr(A)}\big((4a+2d)a_{11}(z)a_{21}(z)+(2a+4d)a_{12}(z)a_{22}(z)\nonumber\\
&&\hspace{-0.3cm} \phantom{e^{-2z\tr(A)}\big(}+(b+c)(a_{11}(z)a_{22}(z)+a_{12}(z)a_{21}(z))\big),\label{Gi}
\end{eqnarray}

\noindent and $W$ is

\begin{eqnarray*}
W(z,p) &=& \sqrt{1+(a_{11}(z)p_1+a_{21}(z)p_2)^2+(a_{12}(z)p_1+a_{22}(z)p_2)^2}\\
&=& \sqrt{1+e^{2z\tr(A)}\left(\vphantom{\big(}Q_{22}(z)p_1^2-2Q_{12}(z)p_1p_2+Q_{11}(z)p_2^2\right)}.
\end{eqnarray*}

Following the above notation, we define the \emph{mean curvature operator}:

\begin{eqnarray}
Q(u)\hspace{-0.3cm} & = &\hspace{-0.3cm} u_{xx}\left(Q_{22}(u)+u_y^2\right)+u_{yy}\left(Q_{11}(u)+u_x^2\right) + 2u_{xy}\left(Q_{12}(u)-u_xu_y\right)\nonumber\\
&&\hspace{-0.3cm} +G_1(u)u_x^2+G_2(u)u_y^2+G_3(u)u_xu_y+(a+d)e^{-2u\tr(A)},\label{Q}
\end{eqnarray}

\noindent and it follows that, if $\Omega \subseteq \raz$, then the $\pi-$graph of a function $u:\Omega \rightarrow \R$ is going to be a minimal surface of $\rar$ if and only if $u$ satisfies $Q(u) = 0$. We also notice that $Q$ is a quasilinear elliptic operator, as the matrix

\begin{equation}\label{MatrixQ}
\mathcal{Q} = \Md{Q_{22}(z)+p_2^2}{Q_{12}(z)-p_1p_2}{Q_{12}(z)-p_1p_2}{Q_{11}(z)+p_1^2}
\end{equation}

\noindent is positive definite for every $z \in \R$ and $p=(p_1,\,p_2) \in \R^2$, as it is easy to see using the relation 

\begin{equation*}
Q_{11}(z)Q_{22}(z)-Q_{12}(z)^2 = e^{-2z\tr(A)}.
\end{equation*}

On the papers of Meeks, Mira, P\'erez and Ros, some work has been developed in order to understand constant mean curvature $\pi-$graphs: the fact that $\rar$ admits a foliation by parallel horizontal planes of constant mean curvature $H = \tr(A)/2$ determines much of the structure of those graphs. For instance, using this property and the mean curvature comparison principle, they are able to prove

\begin{lemma}[Assertion 15.5, \cite{MMPR3}]\label{tracelemma}
Let $D \subseteq \raz$ be a convex compact disk and let $C = \partial D$ be its boundary. Consider $\pi(x,y,z) = (x,y,0)$ the vertical projection. If $\Gamma \subseteq \pi^{-1}(C)$ is a closed simple curve such that the projection $\pi:\Gamma \rightarrow C$ monotonically parameterizes\footnote{This means that $\pi(\Gamma) \subset \partial \Omega$ and $ \pi^{-1}(\{p\}) \cap \Gamma$ is either a single point or a compact interval for every $p \in \partial \Omega$} $C$ and $h:\Gamma \rightarrow \R$ is the height function, let $c_0 = \inf_{\Gamma}h$ and $c_1 = \sup_{\Gamma}h$. If $\Sigma$ is a compact minimal surface with $\partial \Sigma = \Gamma$, it follows:

\begin{enumerate}
 \item If $\tr(A) \geq 0$, then $\Sigma \subseteq \pi^{-1}(D)\cap \{z \geq c_0\}$;
 \item If $\tr(A) \leq 0$, then $\Sigma \subseteq \pi^{-1}(D)\cap \{z \leq c_1\}$.
\end{enumerate}

\end{lemma}

On the particular case of graphs, Lemma \ref{tracelemma} implies that a minimal graph over some smooth domain $\Omega \subseteq \raz$, compact and convex, satisfy the maximum principle if $\tr(A) \leq 0 $ and satisfy the minimum principle if $\tr(A) \geq 0$, satisfying both only on the unimodular case. However, when $\tr(A) > 0$ no uniform upper bound is obtained, neither a lower bound when $\tr(A) <0$. This motivates the search for height estimates for minimal graphs, obtained on the next result. Perhaps, the proof of Theorem~\ref{main} is as interesting as the result itself, as it gives some understanding on the behaviour of the operator $Q$, given by \eqref{Q}, on the many possible settings for the matrix $A$.

\begin{theorem}\label{main}
Let $A\in M_2(\R)$ be a matrix as in \eqref{A} and let $\rar$ be a semidirect product endowed with its canonical left invariant metric. Let $\Omega \subseteq \raz$ be a bounded, convex domain and let $\alpha \in \R$ be any given constant. Then, there exists a constant $C = C(\text{diam}(\Omega),\alpha)$ such that for every $u$ satisfying $Q(u) \geq 0$, it holds that

\begin{equation}\label{eqHE}
\displaystyle \sup_{\Omega} u \leq \max\left\{\sup_{\partial\Omega}u,\,\alpha\right\} + C.
\end{equation}

Furthermore, the \emph{most efficient} choice for $\alpha$ is to take $\alpha_0 = \sup_{\partial \Omega} u$, on the sense that if $\alpha > \alpha_0$, \eqref{eqHE} gives no relation concerning the difference between the maximum $u$ attains on the interior and the maximum it attains on the boundary, and if $\alpha < \alpha_0$ then $C(\alpha) > C(\alpha_0)$, so there is a worsening on the estimate.

\end{theorem}

\begin{remark} We notice that $Q(u) \geq 0$ has the geometric meaning that the $\pi-$graph of $u$ has mean curvature vector pointing upwards. Also, we remark that Theorem~\ref{oscilation} below is where we justify the sharp choice of $\alpha$ as $\alpha_0 = \sup_{\partial\Omega}u$. We chose to enunciate the theorem with a \emph{free} $\alpha$ instead of taking directly $\alpha = \alpha_0$ in order to take limits of $C(\alpha)$ when $\abs{\alpha}\rightarrow + \infty$, with the liberty of having the function $u$ fixed.
\end{remark}

\begin{proof}[Proof of Theorem~\ref{main}]

First, we notice that when $\tr(A) \leq 0$, the result is trivial with $C = 0$ and without the need for an $\alpha$, by Lemma \ref{tracelemma}. Thus we will suppose that $\tr(A) > 0$ and focus on the non unimodular case. Without loss of generality, after a homothety of the metric we may assume that $\tr(A) = 2$ and that $A$ is given by

\begin{equation}\label{Anotuni}
A = \Md{1+a}{b}{c}{1-a},
\end{equation}

\noindent for some $a,\,b,\,c \in \R$. Now, we will divide the proof in two cases, starting when $A$ is not a diagonal matrix:

\vspace{0.2cm}

\noindent \textbf{Case 1.} First, we suppose that $A$ is not a diagonal matrix. We will define a quasilinear operator $R$ related with $Q$ that satisfy the hypothesis of the comparison principle (for instance Theorem 10.1 of \cite{GT}) and find an \emph{ad hoc} positive function $v:\Omega\rightarrow \R$, whose construction will depend only on $\Omega$ and $\alpha$ such that $R(v+\max\{\sup_{\partial \Omega}u,\,\alpha\}) \leq R(u)$. Then our constant $C$ will be simply given by $C = \sup_{\Omega} v$. We begin by proving the following key claim, which will also be used on Section~\ref{secScherk}:

\begin{claim}\label{cl1} Let the functions $Q_{ij}$ be the ones defined on \eqref{Qij} with respect to the matrix $A$ of \eqref{Anotuni}, where either $b\neq 0 $ or $c \neq 0$. Then, there is some $\lambda > 0$ such that at least one of the following hold, for every $z \in \R$:

\begin{itemize}
\item[i.] $Q_{22}(z)e^{2z} > \lambda$;
\item[ii.] $Q_{11}(z)e^{2z} > \lambda$.
\end{itemize}

Moreover, if $a^2 + bc \leq 0$, both i. and ii. hold, and if $a^2+ bc >0$, then $b\neq 0$ is equivalent to i. and $c\neq 0$ is equivalent to ii. 
\end{claim}

\vspace{0.2cm}

\noindent \emph{Proof of Claim \ref{cl1}.} We are going to prove Claim~\ref{cl1} in each of three (family of) possibilities to the exponential of $A$. First, we write $A = I + A_0$, where $I$ is the identity matrix and $A_0$ is the traceless part of $A$ given by

\begin{equation*}
A_0 = \Md{a}{b}{c}{-a}.
\end{equation*}

As $I$ and $A_0$ commute, we obtain that $e^{Az} = e^{Iz+A_0z} = e^{Iz}e^{A_0z}$, thus

\begin{equation*}
e^{Az} = e^z\Md{a_{11}^0(z)}{a_{12}^0(z)}{a_{21}^0(z)}{a_{22}^0(z)},
\end{equation*}

\noindent where we denote by $a_{ij}^0(z)$ the coefficients of the exponential $e^{A_0z}$. In particular, we obtain that $a_{ij}(z) = e^za_{ij}^0(z)$, and it follows that

\begin{equation*}
Q_{11}(z)e^{2z} = e^{-4z}\left[a_{21}(z)^2+a_{22}(z)^2\right]e^{2z} = a_{21}^0(z)^2+a_{22}^0(z)^2,
\end{equation*}

\noindent and analogously 

\begin{equation*}
Q_{22}(z)e^{2z} = a_{11}^0(z)^2+a_{12}^0(z)^2.
\end{equation*}

We observe that the characteristic equation of $A_0$ is given by $0 = \det(A_0-tI) = t^2 - (a^2+bc)$, so if we denote by $d = \sqrt{\vert a^2+bc\vert}$, the exponential of $A_0$ is given by\footnote{We remark that the constant $a^2+bc$ is linked with the Milnor $D-$invariant of $\rar$, which is (following \cite{MP}) defined by $D = \det(A) = 1-(a^2+bc)$. So each case $a^2+bc>0,\,a^2+bc = 0 $ and $a^2+bc< 0 $ is in correspondence with $D < 1,\,D=1$ and $D > 1$, respectively.}

\begin{equation}\label{dmen0}
e^{A_0z} = \Md{\cos(dz)+\frac{a}{d}\sin(dz)}{\frac{b}{d}\sin(dz)}{\frac{c}{d}\sin(dz)}{\cos(dz)-\frac{a}{d}\sin(dz)}, \text{ when }a^2+bc < 0,
\end{equation}

\begin{equation}\label{dig0}
e^{A_0z} = \Md{1+az}{bz}{cz}{1-az}, \text{ when }a^2+bc = 0,
\end{equation}

\begin{equation}\label{dmai0}
e^{A_0z} = \Md{\cosh(dz)+\frac{a}{d}\sinh(dz)}{\frac{b}{d}\sinh(dz)}{\frac{c}{d}\sinh(dz)}{\cosh(dz)-\frac{a}{d}\sinh(dz)}, \text{ when }a^2+bc > 0.
\end{equation}

Now we let $f(z) = a_{11}^0(z)^2+a_{12}^0(z)^2$ and $g(z) = a_{21}^0(z)^2+a_{22}^0(z)^2$ and prove that there is some $\lambda > 0$ such that either $f(z)>\lambda$ or $g(z) > \lambda$. First, we notice that both $f(z)$ and $g(z)$ are always positive, as the existence of some $z_0 \in \R$ such that $f(z_0) = 0$ or $g(z_0) = 0$ would imply that $\det(e^{A_0z_0}) = 0$, an absurdity. Then, we just need to check the behaviour of $f$ and $g$ at $\pm\infty$.

\vspace{0.1cm}

First, if it was $a^2+bc < 0$, the existence of $\lambda$ as claimed follows directly from the fact that both $f$ and $g$ are periodic and positive, by \eqref{dmen0}. If it was $a^2+bc = 0$, then we have that $f$ and $g$ are given by

\begin{eqnarray*}
f(z) & = & (1+az)^2+(bz)^2 = (a^2+b^2)z^2+2az+1\\
g(z) & = & (1-az)^2+(cz)^2 = (a^2+c^2)z^2-2az+1,
\end{eqnarray*}

\noindent both strictly positive at infinity for any choice of $a,\,b,\,c$, so we also have the existence of $\lambda$ on this case. Finally, if $a^2+bc > 0$, $f$ and $g$ would be given by

\begin{eqnarray*}
f(z) & = & \left(\cosh(dz)+\frac{a}{d}\sinh(dz)\right)^2+\left(\frac{b}{d}\sinh(dz)\right)^2\\
g(z) & = & \left(\cosh(dz)-\frac{a}{d}\sinh(dz)\right)^2+\left(\frac{c}{d}\sinh(dz)\right)^2.
\end{eqnarray*}

If $i.$ was not true, either $\lim_{z\rightarrow -\infty} f(z) = 0$ or $\lim_{z\rightarrow +\infty} f(z) = 0$, so it would follow that $b = 0$. Analogously if $\lim_{z\rightarrow -\infty} g(z) = 0$ or $\lim_{z\rightarrow +\infty} g(z) = 0$, we would have $c = 0$. As $A$ is not a diagonal matrix, at least one between $i.$ and $ii.$ is true, finishing the proof of the claim. $\hfill \diamondsuit$

\vspace{0.3cm}

Now we proceed with the proof of the theorem by proving the existence of $\Lambda > 0$ such that $G_{1}(z)\leq \Lambda Q_{22}(z)$ and $G_{2}(z)\leq \Lambda Q_{11}(z)$. Just observe that, by definition,

\begin{eqnarray}
\frac{G_1(z)}{Q_{22}(z)} & = & \frac{e^{-4z}\left[(3+a)a_{11}(z)^2+(3-a)a_{12}(z)^2+(b+c)a_{11}(z)a_{12}(z)\right]}{e^{-4z}\left[a_{11}(z)^2+a_{12}(z)^2\right]}\nonumber\\
& = & 3+a\frac{a_{11}(z)^2-a_{12}(z)^2}{a_{11}(z)^2+a_{12}(z)^2}+(b+c)\frac{a_{11}(z)a_{12}(z)}{a_{11}(z)^2+a_{12}(z)^2}\nonumber\\
& \leq & 3 + \vert a \vert + \frac{\vert b+c \vert}{2} = \Lambda, \label{G1Q22}
\end{eqnarray}

\noindent and, mutatis mutandis, the same estimate holds for the quotient $G_2(z)/Q_{11}(z)$.

\vspace{0.1cm}

Finally we can finish the proof of the theorem when $A$ is not a diagonal matrix, using the existence of $\lambda$ and $\Lambda$ as before. First, we assume that i. holds and let $u:\Omega \rightarrow \R$ be any function that satisfy $Q(u)\geq 0$. Then we define the quasilinear elliptic operator $R$ as

\begin{eqnarray}
R(w)\hspace{-0.3cm} & = &\hspace{-0.3cm} w_{xx}\hspace{-0.1cm}\left(\frac{Q_{22}(u)+w_y^2}{Q_{22}(u)}\right)+ w_{yy}\left(\frac{Q_{11}(u)+w_x^2}{Q_{22}(u)}\right) + 2 w_{xy}\left(\frac{Q_{12}(u)-w_xw_y}{Q_{22}(u)}\right)\nonumber\\
&&\hspace{-0.3cm} + \frac{G_1(u)}{Q_{22}(u)}w_x^2+ \frac{G_2(u)}{Q_{22}(u)}w_y^2+ \frac{G_3(u)}{Q_{22}(u)}w_xw_y+2\frac{e^{-2u}}{Q_{22}(u)}e^{-2w},\label{R1}
\end{eqnarray}

\noindent and notice that $R(u) = Q(u)/Q_{22}(u) \geq 0$.

Now, as $\Omega$ is a bounded domain, after a horizontal translation (which is an isometry of the ambient space) we may suppose without loss of generality that it is contained on a strip

\begin{equation*}
\Omega \subseteq \{(x,y,0) \in \rar;\,1 < x < M\},
\end{equation*}

\noindent for some $M > 1$. We let $v(x,y) = \ln(l x)/L$, where $l,\,L > 0$ are constants yet to be defined. Then, if $\alpha \in \R$ is any a priori chosen number, we have that

\begin{eqnarray*}
R(v+\alpha) & = & v_{xx}+\frac{G_1(u)}{Q_{22}(u)}v_x^2+2\frac{e^{-2u}}{Q_{22}(u)}e^{-2(v+\alpha)}\\
& < & v_{xx} +\Lambda v_x^2 + \frac{2}{\lambda}e^{-2v}e^{-2\alpha}.
\end{eqnarray*}

Then, using that $v_{x} = \frac{1}{Lx}$ and $v_{xx} = \frac{-1}{Lx^2}$, we obtain

\begin{eqnarray}
R(v+\alpha) & < & -\frac{1}{Lx^2}+\Lambda\frac{1}{L^2x^2}+\frac{2}{\lambda e^{2\alpha}}(lx)^{-2/L}\label{Rvalpha}\\
& = & \left.\left.\frac{1}{Lx^2}\right[-1+\frac{\Lambda}{L}+\frac{2L}{\lambda e^{2\alpha}l^{2/L}}x^{(2L-2)/L}\right].\nonumber
\end{eqnarray}

Now, take $L = 1 + \Lambda$. As $1<x<M$, follows that

\begin{equation}\label{Rvalpha2}
R(v+\alpha) < \left.\left.\frac{1}{(1+\Lambda)x^2}\right[-\frac{1}{1+\Lambda}+2\frac{1+\Lambda}{\lambda e^{2\alpha}l^{\frac{2}{1+\Lambda}}}M^{\frac{2\Lambda}{1+\Lambda}}\right],
\end{equation}

\noindent and then we just choose $l$ big enough (in particular we may assume $l\geq 1$, so $v> 0 $) such that

\begin{equation}\label{defC}
-\frac{1}{1+ \Lambda} + 2\frac{1+\Lambda}{\lambda e^{2\alpha}l^{\frac{2}{1+\Lambda}}}M^{\frac{2\Lambda}{1+\Lambda}} < 0,
\end{equation}

\noindent so $R(v+\alpha) < 0$. We remark that the choice of $l$ and $L$ as above depends only on $\lambda, \Lambda,\,\alpha$ and $M$, so it does not depend on $u$. Let 

$$v_0 = v+\max\left\{\sup_{\partial \Omega } u,\,\alpha\right\},$$

\noindent so $v_0 \geq v+\alpha$, and it follows from the definition of $R$ (and from the fact that $v_0-v$ is constant) that

\begin{equation*}
R(v_0) \leq R(v+\alpha) < 0 \leq R(u),
\end{equation*}

\noindent then, as $R$ satisfies the hypothesis of the Comparison Principle and $v_0 \geq u$ on $\partial \Omega$, follows that $\sup_{\Omega} u \leq \sup_{\Omega} v_0$. Finally, we set $C = \sup_{\Omega} v$, and the theorem follows when $A$ is not diagonal and i. holds. The other cases will follow analogously, but we write the main steps of their proofs.

\hspace{0.1cm}

Still on Case 1 of $A$ not being a diagonal matrix, if i. was not true, then $b = 0$, $c \neq 0$ and we assume ii. to define

\begin{eqnarray}
R(w)\hspace{-0.3cm} & = & \hspace{-0.3cm} w_{xx}\hspace{-0.1cm}\left(\frac{Q_{22}(u)+w_y^2}{Q_{11}(u)}\right)+ w_{yy}\left(\frac{Q_{11}(u)+w_x^2}{Q_{11}(u)}\right) + 2 w_{xy}\left(\frac{Q_{12}(u)-w_xw_y}{Q_{11}(u)}\right)\nonumber\\
&& \hspace{-0.3cm} + \frac{G_1(u)}{Q_{11}(u)}w_x^2+ \frac{G_2(u)}{Q_{11}(u)}w_y^2+ \frac{G_3(u)}{Q_{11}(u)}w_xw_y+2\frac{e^{-2u}}{Q_{11}(u)}e^{-2w}.\label{R2}
\end{eqnarray}

From here, just proceed analogously as on the case where $b \neq 0$, by letting $v(x,y) = \ln(ly)/L$, and this finishes the proof on the first case.

\vspace{0.2cm}

It remains to prove the theorem when $A$ is diagonal (then neither i. nor ii. hold), so $A$ is of the form

\begin{equation*}
A = \Md{1+a}{0}{0}{1-a},
\end{equation*}

\noindent then $a_{11}(z) = e^{(1+a)z}$, $a_{22}(z) = e^{(1-a)z}$ and $a_{12}(z) = a_{21}(z) = 0$. It follows that the operator $Q$ is given by

\begin{eqnarray*}
Q(u) & = & u_{xx}\left(e^{-2(1-a)u}+u_y^2\right)+u_{yy}\left(e^{-2(1+a)u}+u_x^2\right) - 2u_{xy}\left(u_xu_y\right)\\
&&+(3+a)e^{-2(1-a)u}u_x^2+(3-a)e^{-2(1+a)u}u_y^2+2e^{-4u}.
\end{eqnarray*}

\noindent If $a \geq 0$ we define $R$ as the operator 

\begin{eqnarray}
R(w)\hspace{-0.3cm} & = &\hspace{-0.3cm}w_{xx}\hspace{-0.1cm}\left(1+e^{2(1-a)u}w_y^2\right)\hspace{-0.1cm}+\hspace{-0.1cm}w_{yy}\hspace{-0.1cm}\left(e^{-4au}+e^{2(1-a)z}u_x^2\right) - 2w_{xy}\hspace{-0.1cm}\left(e^{2(1-a)z}w_xw_y\right)\nonumber\\
&&\hspace{-0.3cm}+(3+a)w_x^2+(3-a)e^{-4au}w_y^2+2e^{-2(1+a)w},\label{R3}
\end{eqnarray}

\noindent and if $a < 0$, $R$ will be defined as

\begin{eqnarray}
R(w)\hspace{-0.3cm} & = & \hspace{-0.3cm}w_{xx}\hspace{-0.1cm}\left(e^{4au}+e^{2(1+a)u}w_y^2\right)\hspace{-0.1cm}+\hspace{-0.1cm}w_{yy}\hspace{-0.1cm}\left(1+e^{2(1+a)u}w_x^2\right) - 2w_{xy}\hspace{-0.1cm}\left(e^{2(1+a)u}w_xw_y\right)\nonumber\\
&&\hspace{-0.3cm}+(3+a)e^{4au}w_x^2+(3-a)w_y^2+2e^{-2(1-a)w}.\label{R4}
\end{eqnarray}

Now, we just set $v$ to be again $v(x,y) = \ln(lx)/L$ when $a\geq0$ and $v(x,y) = \ln(ly)/L$ when $a < 0$ and, as the term on both operators which contains no derivative terms in decreasing on $u$, the proof will follow as in the previous case, using $\Lambda = 3+\vert a \vert$ and $\lambda = 1$.\end{proof}

On the next theorem we prove that the dependence on $\alpha$ cannot be removed, on the sense that the existence of a constant which does not depend on $\alpha$ is not possible. Precisely, we prove:

\begin{theorem}\label{thmabs}
Let $A$ be a matrix as in \eqref{Anotuni} and let $X = \rar$ be a non-unimodular semidirect product endowed with its canonical left invariant metric. Let $\Omega \subseteq \raz$ be a bounded, convex domain. Then, for every constant $C >0$ there exists some function $u:\Omega \rightarrow \R$ satisfying $Q(u) = 0$ and also

\begin{equation}\label{eqthmabs}
\sup_{\Omega} u > \sup_{\partial \Omega} u + C.
\end{equation}

\end{theorem}

The proof of Theorem \ref{thmabs} above is by contradiction and consists on using the vertical translation that rises from the group structure to translate a family of solutions tending to $-\infty$, all to height $0$. We prove that if Theorem \ref{thmabs} was false, such family would be uniformly bounded, and this would generate a contradiction with the following theorem, due to Meeks, Mira, P\'erez and Ros \cite{MMPR3}:

\begin{theorem}[Theorem 15.4, \cite{MMPR3}]\label{15.4} Let $X$ be a non-unimodular metric Lie group which is isomorphic and isometric to a semidirect product $\rar$, $A \in M_2(\R)$. Suppose that $\Gamma(n) \subseteq \raz$ is a sequence of $C^2$ simple closed convex curves with $e = (0,0,0) \in \Gamma(n)$ such that the geodesic curvatures of $\Gamma(n)$ converge uniformly to $0$ and the curves $\Gamma(n)$ converge on compact subsets to a line $L$ with $e \in L$ as $n\rightarrow \infty$. Then, for any sequence $M(n)$ of compact branched minimal disks with $\partial M(n) = \Gamma(n)$, the surfaces $M(n)$ converge $C^2$ on compact subsets as $n \rightarrow \infty$ to the vertical half plane $\pi^{-1}(L)\cap \left[\R^2\rtimes_A[0,\,\infty)\right]$.
\end{theorem}

\begin{proof}[Proof of Theorem \ref{thmabs}]

We begin by proving the following claim:

\begin{claim}\label{cl2} Let $\sn1 = \{(x,\,y)\in \R^2;\,x^2+y^2=1\}$ be the unit circle centred on the origin of $\R^2$. Let $A \in M_2(\R)$ be a matrix with $\tr(A) = 2$, as in \eqref{Anotuni}, and let $e^{Az}$ be its exponential map. Then there is a point $p \in \sn1$ and an increasing sequence $z_n \in (0,\,+\infty)$ such that $\Gamma_n = \left.\left.e^{Az_n}\right(\sn1-p\right)$ satisfies the hypothesis of Theorem \ref{15.4} at the origin, i.e., the geodesic curvature of $\Gamma_n$ at $0$ is converging to zero and $\Gamma_n$ is converging to a line $L$ on compacts, with $0\in L$.
\end{claim}

\noindent \emph{Proof of Claim \ref{cl2}.} We again denote by $A_0$ the traceless part of $A$ and observe that $e^{Az}= e^{z}e^{A_0z}$. Then we have that $e^{Az}\sn1 = e^z\left(e^{A_0z}\sn1\right)$ is a homothety by $e^z$ of the curve $e^{A_0z}\sn1$. Now we let $d = \sqrt{\vert a^2+bc\vert}$ and divide the proof on the three aforementioned cases given by equations \eqref{dmen0}, \eqref{dig0} and \eqref{dmai0}.

\vspace{0.1cm}

First, if $a^2+bc<0$, we let $p \in \sn1$ be any point and define $z_n = \frac{2n\pi}{d}$. Then $e^{A_0z_n} = \text{Id}$, so $e^{Az_n}\sn1$ is a circle of radius $e^{2z_n}$ centred at the origin, and $\Gamma_n = e^{Az_n}(\sn1 - p)$ is a circle through the origin with radius $e^{2z_n}$. As $z_n\rightarrow \infty$, $\Gamma_n$ will converge to a line $L$ through $0$ and the claim is proved on this case.

\vspace{0.1cm}

Secondly, if $a^2+bc = 0$, then $e^{A_0z}$ is given by \eqref{dig0} and $e^{A_0z}\sn1$ is an ellipse and the homotheties of an ellipse by $e^{n}$ admits a point where its geodesic curvature converges to zero and, after a translation, it converges to a line on compacts, proving the claim on the second case.

\vspace{0.1cm}

Finally, if $a^2+bc > 0$, $e^{A_0z}$ is given by \eqref{dmai0}. If $bc \neq 0$, then $d \neq \abs{a}$, and if $z$ is big enough we have that $\cosh(dz) \simeq e^{dz}/2$ and $\sinh(dz) \simeq e^{dz}/2$, so 

$$e^{A_0z} \simeq \frac{e^{dz}}{2d}\Md{d+a}{b}{c}{d-a},$$

\noindent and $e^{Az}\sn1$ is asymptotic to a homothety of $e^{(d+2)z}$ of an ellipse, which has the desired properties. The last case to be treated is when $d^2 = a^2+bc = a^2 > 0$, then

\begin{equation*}
e^{A_0z} = \Md{e^{dz}}{\frac{b}{d}\sinh(dz)}{\frac{c}{d}\sinh(dz)}{e^{-dz}}\simeq \frac{e^{dz}}{d}\Md{d}{b}{c}{de^{-2dz}},
\end{equation*}

\noindent and, for $z$ large enough it follows that $e^{A_0z}\sn1$ is asymptotic to a line segment, with multiplicity 2. Now, it depends on the two possible cases $0< d \leq 1$ or $d > 1$ to understand what is the convergence of $e^{Az}\sn1$: if $d \leq 1$, then the homothety of $e^z$ on $e^{A_0z}$ will \emph{open up} the segment and make it asymptotic to an ellipse again, which again admits a point $p$ as claimed. If $d > 1$, then the action of $e^z$ still makes $e^{Az}\sn1$ converge to a line on compacts, so the claim is proved. $\hfill\diamondsuit$

\vspace{0.2cm}

The proof of Theorem \ref{thmabs} will now be done arguing by contradiction. Suppose that for some smooth bounded domain $\Omega \subseteq \raz$ there is a $C > 0$ such that for every solution of $Q(u) = 0$, it holds that

\begin{equation}\label{absurdo}
\displaystyle\sup_{\Omega} u \leq \sup_{\partial\Omega} u + C.
\end{equation} 

In particular, the same estimate holds for any bounded, smooth domain contained on $\Omega$. Now, let $r>0$ be such that an euclidean ball $B_r$ with radius $r$ is contained on $\Omega$. Let $\sn1(r) = \partial B_r$ be the circle that bounds $B_r$ and let $p\in \sn1(r)$ and $(z_n)_{n\in \mathds{N}}$ be the ones given on by Claim 2. We consider, for each $n \in \mathds{N}$, the problem

\begin{equation}\label{Pn}
\left\{
\begin{array}{cc}
Q(u) = 0 & \text{ on } B_r\\
u = - z_n & \text{ on }\partial B_r.
\end{array}\right.
\end{equation}

\noindent The existence result due to Meeks-Mira-P\'erez-Ros, Theorem 15.1 of \cite{MMPR3}, implies that \eqref{Pn} admits a solution $u_n:B_r \rightarrow \R$, and, from equation \eqref{absurdo}, follows that, for every $n \in \mathds{N}$, the function $u_n$ satisfies

\begin{equation*}
\displaystyle\sup_{B_r} u_n \leq \sup_{\partial \Omega}u_n + C =   -z_n+ C.
\end{equation*} 

We will translate the functions $u_n$ vertically, all to height $0$, using the left translation of the group $L_{(0,0,z_n)}$ to obtain a contradiction. If $\Sigma_n = \text{graph}(u_n)$, we notice that

\begin{eqnarray*}
L_{(0,0,z_n)}\Sigma_n & = & \left\{L_{(0,0,z_n)}\left(x,y,u_n(x,y)\right); \, (x,\,y) \in B_r\vphantom{\begin{array}{c}x\\y\end{array}}\right\}\\
& = & \left\{ \left(e^{Az_n}\left(\begin{array}{c}x\\y\end{array}\right),u_n(x,y) + z_n\right),\,(x,\,y) \in B_r\right\}\\
& = & \left\{ \left(\xt,\,\yt,u_n\left(e^{-Az_n}\left(\begin{array}{c}\xt \\ \yt \end{array}\right)\right) + z_n\right),\,(\xt,\,\yt) \in e^{Az_n}B_r\right\}.
\end{eqnarray*}

If we let $v_n:e^{Az_n}B_r \rightarrow \R$ be the function given by 

\begin{equation*}
v_n(x,y) = u_n\left(e^{-Az_n}\left(\begin{array}{c}x \\ y \end{array}\right)\right) + z_n,
\end{equation*}

\noindent it follows that the graph of $v_n$ is a left translate of the graph of $u_n$, in particular its graph $\Sbn = L_{(0,0,z_n)}\Sigma_n$ is a minimal surface of $\rar$. However, these graphs $\Sbn$ satisfy the hypothesis of Theorem \ref{15.4}, thus they should converge, in compact sets, to a vertical half plane, but we notice that

\begin{equation*}
\displaystyle \sup_{e^{Az_n}B_r} v_n = \sup_{B_r}u_n +z_n \leq C,
\end{equation*}

\noindent so the sequence $v_n$ is uniformly bounded, generating a contradiction.\end{proof}

We notice that, on last proof we show more than the existence of \emph{a} function $u$ as on \eqref{eqthmabs} for a fixed constant $C$. We actually proved that \emph{any} sequence of functions with values along the boundary converging to $- \infty$ cannot have uniformly bounded oscillation. In particular, as a consequence of this result and of the proof of Theorem~\ref{main}, we can close this section with the following result:

\begin{theorem}\label{oscilation}
Let $X$ be a metric Lie group which is isomorphic and isometric to a semidirect product $\rar$ for some $A \in M_2(\R)$. Let $\Omega \subseteq \raz$ be some open, bounded, smooth domain, $k \in \mathds{Z}$ be given and let $u_k$ denote a solution to the problem

\begin{equation}\label{P}\tag{$P_k$}
\left\{
\begin{array}{cc}
Q(u) = 0 & \text{ in } \Omega\\
u = k & \text{ on } \partial \Omega.
\end{array}
\right.
\end{equation}

Then, either $X$ is an unimodular metric Lie group and $u_k \equiv k$ is the constant function, or $X$ is non-unimodular, $\tr(A) \neq 0$ and $u_k$ satisfies:

\begin{itemize}
\item $u_k > k$ in $\Omega$, $\displaystyle \lim_{k\rightarrow -\infty}\text{osc}_{\Omega}(u_k) = + \infty$ and $\displaystyle \lim_{k\rightarrow +\infty}\text{osc}_{\Omega}(u_k) = 0$, when $\tr(A) > 0$;
\item $u_k < k$ in $\Omega$, $\displaystyle \lim_{k\rightarrow -\infty}\text{osc}_{\Omega}(u_k) = 0$ and $\displaystyle \lim_{k\rightarrow +\infty}\text{osc}_{\Omega}(u_k) = +\infty$, when $\tr(A) < 0$,
\end{itemize}

\noindent where $\text{osc}_\Omega (u) =\sup_\Omega(u) - \inf_{\Omega}(u)$ denotes the oscillation of a function $u$ in $\Omega$.

\end{theorem}

\begin{proof}
If $\tr(A) = 0$, it is clear that $u_k\equiv k$ is the unique solution to \eqref{P}, by Lemma~\ref{tracelemma}. Also, as the change $A\mapsto -A$ corresponds to a simple change of orientation $z \in \rar \mapsto -z \in \R^2\rtimes_{-A}\R$, we can simply prove the case of $\tr(A) > 0$, and, as previous, it is without loss of generality that we assume that $\tr(A) = 2$, so $A$ is written as on \eqref{Anotuni}.

From Lemma \ref{tracelemma}, it follows that $u_k \geq k$ on $\Omega$, and, if at an interior point $x \in \Omega$ the function $u_k$ attains its minimum $u_k(x) = k$, then the mean curvature comparison principle, applied to $\Sigma_k = \graph(u_k)$ and to the plane $\{z = k\}$ will imply that the mean curvature of $\Sigma_k$ is at least as big as the one of the plane, which is $1 > 0$, a contradiction that proves that, on the interior of $\Omega$, it holds $u_k > k$.

\vspace{0.2cm}

The second part of the claim follows like on the proof of Theorem \ref{thmabs}: if the oscillation of $u_k$ was not going to $+ \infty$ when $k\rightarrow - \infty$ then we could translate all the minimal surfaces $\Sigma_k = \graph(u_k)$ to height zero and obtain a contradiction with Theorem \ref{15.4}.

\vspace{0.2cm}

In order to obtain the last part of the Theorem, that the oscillation of $u_k$ goes to zero when $k$ approaches $+ \infty$, we recall the proof of Theorem \ref{main}: we obtained a constant $C$ depending on many parameters, $C = C(l,L,\lambda,\Lambda,M,\alpha)$. However, the constants $\lambda$ and $\Lambda$ were depending only on the ambient space, as they came from estimates to the coefficients of the operator $Q$. The constant $M$ was depending uniquely on the diameter of $\Omega$, so it was fixed from the beginning, together with $\Omega$. The free parameters we could work with were $l$ and $L$, depending on the previous ones and on the a priori constant $\alpha$. Using an appropriate choice of $l$ and $L$, we obtained that the constant claimed on the Theorem was

\begin{equation*}
C = \frac{\ln(lM)}{L}.
\end{equation*}

The key steps to chose $l$ and $L$ were between equations \eqref{Rvalpha}, \eqref{Rvalpha2} and \eqref{defC}, but the way we proceeded was thinking on the worst case, where the number $\alpha$ was a \emph{negatively large} number, so we began by choosing $L$ and then got to the definition of a $l$ big enough, in order to compensate $e^{2\alpha}$, close to zero. Now, we are taking $\alpha_k = k$ to be \emph{positive} and \emph{very large}, so we are going to follow a different approach. We begin letting $L = \Lambda + j$, where $j \in \mathds{N}$ is yet to be chosen, and take immediately $l =1$, to obtain, similarly to \eqref{Rvalpha2}, the inequality

\begin{equation}\label{Rvalpha3}
R(v+k) < \left.\left. \frac{1}{(\Lambda + j)x^2}\right[ -\frac{j}{\Lambda+j} + 2\frac{\Lambda+j}{\lambda e^{2k}}M^{\left(2-\frac{2}{\Lambda+j}\right)}\right].
\end{equation}

Then, we proceed as before, and try to find some $j\in \mathds{N}$ such that the right hand side of \eqref{Rvalpha3} becomes negative. Such $j$ is going to exist if and only if

\begin{equation}\label{equivR}
\frac{(\Lambda+j)^2}{jM^{\frac{2}{\Lambda+j}}} < \frac{\lambda }{2M^2}e^{2k}.
\end{equation}

If $k$ is small, maybe there is no $j\in \mathds{N}$ such that \eqref{equivR} holds, but for some $k_0 \in \mathds{N}$ big enough it is possible to find some $j \in \mathds{N}$ satisfying \eqref{equivR} (therefore also \eqref{Rvalpha3}). As the right hand side grows with $k$, for every $k \geq k_0$ there will exist such $j$, and we denote $j(k)$ the \emph{largest} $j\in \mathds{N}$ such that \eqref{equivR} hold (as the left hand side is unbounded with $j$ this is well defined). By taking $L = \Lambda + j(k)$, we use \eqref{Rvalpha3} to obtain that exists a constant $C(k) = C(\Omega,k)$ given by

\begin{equation*}
C(k) = \frac{\ln(M)}{\Lambda + j(k)}
\end{equation*}

\noindent such that $\sup_{\Omega}(u) \leq \max\{\sup_{\partial \Omega} u,k\} + C(k)$, for every $u:\Omega \rightarrow \R$ with $Q(u) \geq 0$, the same result as on Theorem \ref{main} but for a different constant $C$, and only for $k\geq k_0$. It follows, in particular, that the functions $u_k$ satisfy, for $k$ large enough, that

\begin{equation*}
\sup_{\Omega} u_k \leq k + C(k),
\end{equation*}

\noindent so, as $\inf_{\Omega}u_k = k$, we obtain that $\text{osc}_\Omega(u_k) \leq C(k)$.

Finally, as the right hand side of \eqref{equivR} is unbounded with respect to $k$ and the left hand side is also unbounded with $j$, follows that $\lim_{k\rightarrow \infty}j(k) = \infty$, so $C(k) \rightarrow 0$ when $k \rightarrow \infty$, and so it does the oscillation of $u_k$. \end{proof}

We can apply the same argument as above for functions which are not constant along the boundary, to estimate the maximum height it can attain with respect to its maximum along the boundary:

\begin{corollary}\label{corOsc}
Let $\rar$ be a non unimodular semidirect product with $\tr(A) > 0$. Let, for $L > 0$, $u_L:\Omega \rightarrow \R$ be a function satisfying

\begin{equation*}
\left\{\begin{array}{l}
Q(u) \geq 0 \text{ on } \Omega,\\
\sup_{\partial \Omega} u = L.
\end{array}\right.
\end{equation*}

\noindent Then

\begin{equation}
\lim_{L\rightarrow +\infty} \left(\sup_{\Omega} u_L - L\right) = 0.
\end{equation}

\end{corollary}

\section{Scherk-like fundamental pieces}\label{secScherk}

On this section, we use the tools developed on this section together with Killing graphs techniques to obtain an existence result of \emph{Scherk-like 
fundamental pieces}, which are minimal $\pi$-graphs on $\rar$ assuming the value $0$ along a piecewise smooth curve $\gamma \subset \raz$ and having $
\gamma\cup (\{p_1\}\hspace{-0.1cm}\times\hspace{-0.1cm}[0,\,\infty)) \cup (\{p_2\}\hspace{-0.1cm}\times\hspace{-0.1cm}[0,\,\infty))$ as boundary, where 
$p_1$ and $p_2$ are the endpoints of $\gamma$.

\vspace{0.1cm}

On the ambient space of an \emph{unimodular} group $\rar$, A. Menezes \cite{Men} proved the existence of \emph{complete} (without boundary) minimal surfaces, similar to the singly and to the doubly periodic Scherk minimal surfaces of $\R^3$. We would like to take a moment to give the main steps of the proof of Menezes to the existence of a doubly periodic example:

\begin{proof}[Sketch of the proof of Theorem 2, \cite{Men}] Let $\Delta\subseteq \raz$ to be a triangle with vertexes 
\begin{equation*}
o  =  (0,\,0,\,0),\ p_1  =  (a,\,0,\,0),\ p_2  =  (0,\,a,\,0),
\end{equation*}

\noindent for some $a > 0$. Let $P_c$ be the polygon given by the union of segments 
\begin{equation*}
P_c = \overline{\vphantom{p_1(c)}op_1}\cup\overline{p_1p_1(c)}\cup\overline{p_1(c)p_2(c)}\cup\overline{p_2(c)p_2}\cup\overline{\vphantom{p_1(c)}p_2o},
\end{equation*}

\noindent where $p_1(c) = (a,0,c)$ and $p_2(c) = (0,a,c)$. Then, Theorem 15.1 of \cite{MMPR3} implies the existence of a minimal $\pi-$graph $\Sigma_c$ with $\partial \Sigma_c = P_c$.

Then, one key property was observed: $\Sigma_c$ is a \emph{Killing graph} over the vertical domain $\Omega_c = \{(t,a-t,s);0\leq t \leq a,\ 0\leq s \leq c\}$ with respect to the horizontal Killing field $\partial_x+\partial_y$, thus it is \emph{unique}. This implies it is \emph{stable} and that the variation $c \mapsto \Sigma_c$ is continuous. By making $c \rightarrow \infty$, and using curvature estimates due to H. Rosenberg, R. Souam and E. Toubiana \cite{RST} for stable surfaces on homogeneous manifolds, it is possible to show the convergence of $\Sigma_c$ to some surface $\Sigma_\infty$, nowhere vertical and with boundary
\begin{equation*}
\partial \Sigma_{\infty} = P_\infty = \overline{op_1}\cup\big(\{p_1\}\times[0,\,\infty)\big) \cup \overline{op_2} \cup \big(\{p_2\}\times[0,\,\infty)\big).
\end{equation*}

Finally, use the ambient isometries to rotate $\Sigma_\infty$ along the two segments $\overline{op_1}$ and $\overline{op_2}$ to obtain a complete minimal $\pi-$graph on $\rar$, which can be extended periodically by horizontal translations.\end{proof}

On this subject, our contribution is an extension of the above result to any semidirect product $\rar$. Although on the general case it is not possible to find examples with no boundary, on the above special case treated by A. Menezes we reobtain the same result with a different technique when taking limits. Precisely, we prove:

\begin{theorem}\label{thmScherks}
Let $\rar$ be a semidirect product, where $A \in M_2(\R)$ is any matrix with $\tr(A) \geq 0$. Then, there is $L_0 = L_0(\tr(A),\,\det(A)) > 0$ (and $L_0 = \infty$ when $\tr(A) = 0$) such that if $p_1,\,p_2 \in \raz$ satisfy $d(p_1,\,p_2) \leq L_0$, then for any piecewise smooth curve $\gamma\subseteq \raz$ with endpoints $p_1,\,p_2$ which is a convex graph over the segment $\alpha = \overline{p_1p_2}$ and meets $\alpha$ on angles less than $\pi/2$, there exists a minimal surface $\Sigma$ which is a $\pi-$graph and with boundary 

\begin{equation*}
\partial \Sigma = \gamma \cup (\{p_1\}\times[0,\,+\infty))\cup(\{p_2\}\times[0,\,+\infty)).
\end{equation*}

\noindent Moreover, $\Sigma$ is nowhere vertical, is the unique minimal surface on $\rar$ with such boundary and it is a Killing graph over the vertical domain $\Omega_\infty = \alpha\times[0,\,+\infty)$.

\end{theorem}

The proof of Theorem \ref{thmScherks} is given on Section \ref{secProofScherk}. If $\tr(A) > 0$, when considering polygons as $P_c$ above, there is a minimal $\pi-$graph $\Sigma_c$ with boundary $P_c$. However, as the maximum principle does not hold, there is no reason for it to be a Killing graph over $\Omega_c$ and we do not have the tools to ensure the continuity of the family $\Sigma_c$, which makes impossible to use geometric barriers. It becomes clear that, when $\tr(A) \neq 0$, another sequence of surfaces $\Sigma_c$ should be constructed, or other tools (such as stability of minimal $\pi-$graphs) developed.

Our approach will be as follows: instead of considering minimal $\pi-$graphs over a domain on $\raz$, we will look to the problem \emph{horizontally}, and consider an exhaustion of the \emph{half-strip} $\Omega_\infty = \alpha \times[0,\,+\infty)$ by subdomains $\Omega_c$ on which is possible to find a family of minimal Killing graphs with prescribed boundary. Then, we use techniques from Killing graphs and elliptic partial differential equations to ensure the convergence of such family to another minimal Killing graph $\Sigma_\infty$. Then, we go back to the problem \emph{vertically} (as the intermediate Killing graphs will also be $\pi-$graphs, by a result of Meeks, Mira, P\'erez and Ros), and then we apply the geometric barriers developed by A. Menezes to see that the surface $\Sigma_\infty$ is, as claimed, a $\pi-$graph, nowhere vertical.

\subsection{A good exhaustion of $\Omega_\infty$}\label{secKilling}

The next proposition will be of fundamental importance on the construction described on last section, as it will give the exhaustion of $\Omega_\infty$ by domains $\Omega_c$ where is possible to solve the existence of minimal Killing graphs with prescribed boundary (see Figure \ref{OmInfty}).

\begin{figure}
\centering
\begin{subfigure}[b]{0.48\textwidth}
\includegraphics[width=\textwidth]{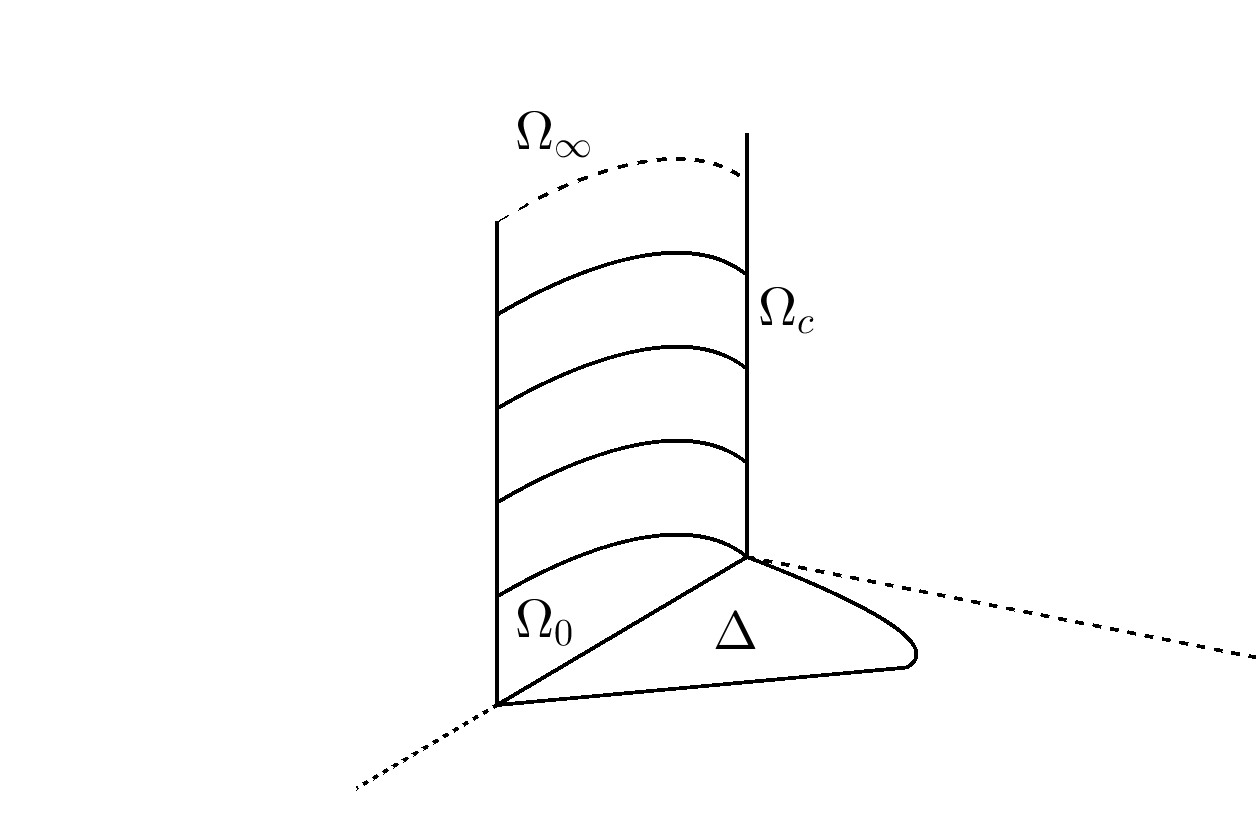}
\end{subfigure}
\begin{subfigure}[b]{0.48\textwidth}
\includegraphics[width=\textwidth]{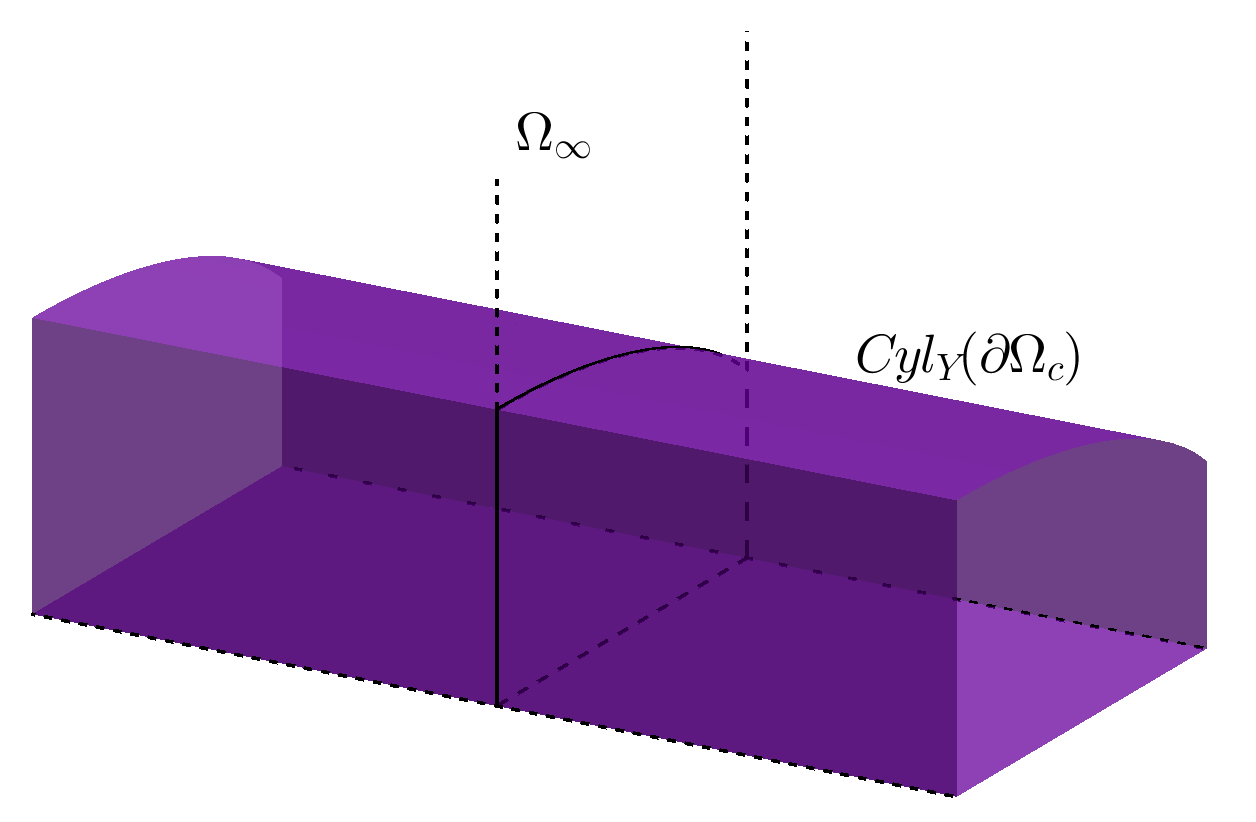}
\end{subfigure}
\caption{The horizontal domain $\Delta$ and the exhaustion of $\Omega_\infty$ by subdomains $\Omega_c$ whose Killing cylinder (on the right) have mean curvature vector pointing inwards.\label{OmInfty}}
\end{figure}

\begin{proposition}\label{propExhaus}
Let $\rar$ be a semidirect product where $\tr(A) \geq 0$. Let $p_1,\,p_2 \in \raz$ and $\alpha = \overline{p_1p_2}$ to be the segment joining $p_1$ and $p_2$. We define the vertical domain

\begin{equation}\label{omegainf}
\Omega_\infty = \alpha\times [0,\,+ \infty).
\end{equation}

Then, there exists some $L_0 = L_0(A) > 0$ such that if $L = \len(\alpha) < L_0$, $\Omega_\infty$ admits a continuous exhaustion $\{\Omega_c\}_{c> 0}$ by domains $\Omega_c$ with boundary given by $\alpha$, a graph over $\alpha$, called $\alpha_c$, and the two vertical segments joining the endpoints of $\alpha$ and $\alpha_c$. This exhaustion is such that the Killing cylinder over $\partial \Omega_c$ with respect to the horizontal Killing field $Y_\theta = \sin(\theta)\partial_x+\cos(\theta)\partial_y$ has mean curvature vector pointing inwards, where $\theta$ is such that $Y_\theta$ is perpendicular to $\Omega_\infty$ at $z = 0$.

\end{proposition}

\begin{proof}

First, we notice that, after a rotation on $A$ as on \eqref{isoiso} and a horizontal translation on $\rar$, which is an ambient isometry, we assume without loss of generality that $p_1 = (0,0,0)$ and $p_2 = (L,0,0)$ for some $L > 0$, then $\alpha$ becomes the segment $\alpha = \{(x,\,0,\,0);\ 0\leq x \leq L\}$ and $\Omega_\infty$ is the half-strip

\begin{equation}
\Omega_\infty = \{(x,\,0,\,z)\in \rar;\ 0\leq x \leq L,\ z \geq 0\},
\end{equation}

\noindent transversal to the Killing field $Y = \partial_y$. Such assumptions will be kept until the end of the paper.

If $\tr(A) = 0$, then the result is trivial (and without the need for an upper bound $L_0$) by taking $\alpha_c$ to be the translate of $\alpha$ to height $c$, $\alpha_c = \{(x,0,c);0\leq x \leq L\}$, as horizontal planes are minimal. Then, until the end of the proof we will treat the non-unimodular case and again we assume without loss of generality that $\tr(A) =2$, so $A$ is a matrix as on \eqref{Anotuni}. We will exhibit the curves $\alpha_c$ explicitly, then we prove they have the desired properties.

First, we treat the case where $A$ is not diagonal and either $a^2+bc \leq 0$ or $b\neq 0$: let $\lambda,\,\Lambda$ the constants related with the matrix $A$ via \emph{i.} of Claim~\ref{cl1} and \eqref{G1Q22}. We let 

\begin{equation*}
L_0 = \sqrt{\frac{\lambda}{2\Lambda}}\frac{\pi}{2}
\end{equation*}

\noindent and, if $L \leq L_0$, we let $f:[0,\,L]\rightarrow \R$ to be given by

\begin{equation}\label{deff}
f(x) = \frac{1}{\Lambda} \ln\left(\frac{\cos\left(\sqrt{\frac{2\Lambda}{\lambda}}x\right)}{\cos\left(\sqrt{\frac{2\Lambda}{\lambda}}L\right)}\right).
\end{equation}

First, we notice that, as $L \leq L_0$, on the interval $[0,L]$ the function $x \mapsto \cos(\sqrt{2\Lambda/\lambda}x)$ is decreasing, so $f$ is well defined and non negative, with $f(x) = 0 \iff x = L$. We let, for $c \geq 0$, $f_c = f + c$ and let $\alpha_c = \text{graph}(f_c) \subseteq \Omega_{\infty}$. When we set 
\begin{equation}
\Omega_c = \left\{(x,0,z) \in \rar;\ 0\leq x \leq L, 0 \leq z \leq f_c(x)\right\},
\end{equation}

\noindent it follows that $\{\Omega_c\}_{c>0}$ is a continuous exhaustion of $\Omega_\infty$. Now we are going to show that the Killing cylinder of the boundary of $\Omega_c$ with respect to $\partial_y$ has mean curvature vector pointing inwards.

The $\partial_y$-Killing cylinder of $\partial \Omega_c$ has four smooth components (see Figure \ref{OmInfty}, right): one is a subdomain of a horizontal plane, so it has mean curvature $1$ pointing upwards, two are contained on vertical planes, thus are minimal. The last component is the one corresponding to $\alpha_c$, and it is a $\pi-$graph of the function $u_c(x,y) = f_c(x)$, on the horizontal strip $\{(x,\,y,\,0) ; \ 0 \leq x \leq L\}$, and \eqref{propPDE} implies that its mean curvature is given by

\begin{equation}
H =\left.\left. \frac{e^{4f_c}}{2W^3}\right[Q_{22}(f_c)f_c^{\prime\prime} + G_1(f_c)\left(f_c^{\prime}\right)^2 + 2 e^{-4f_c}\right].
\end{equation}

Now, we follow the steps on the proof of Theorem \ref{main}: as $b\neq 0$, Claim \ref{cl1} implies that $Q_{22}(z) > \lambda e^{-2z}$, and $G_1/Q_{22} \leq \Lambda$. These relations imply that (as the derivatives of $f_c$ coincide with the ones of $f$)

\begin{equation}\label{Hsemisemifinal}
H \leq \left.\left.\frac{e^{4f_c}}{2W^3}Q_{22}(f_c)\right[ f^{\prime\prime} + \Lambda \left(f^{\prime}\right)^2 + 2 \frac{e^{-2f_c}}{\lambda}\right],
\end{equation}

\noindent whenever $A$ is not diagonal and satisfies either $b\neq 0$ or $a^2+bc \leq 0$. In particular, as $f_c \geq 0$ we have

\begin{equation}\label{Hsemifinal}
H \leq \left.\left.\frac{e^{4f_c}}{2W^3}Q_{22}(f_c)\right[ f^{\prime\prime} + \Lambda \left(f^{\prime}\right)^2 +  \frac{2}{\lambda}\right].
\end{equation}

Now, we observe that $f$ was chosen in such a way it satisfies the ODE 

\begin{equation}\label{ODEf}
f^{\prime\prime} + \Lambda \left(f^{\prime}\right)^2 +  \frac{2}{\lambda} = 0,
\end{equation}

\noindent so, by applying \eqref{ODEf} on \eqref{Hsemifinal}, we obtain that $H \leq 0$, with respect to the upward orientation, so the mean curvature vector of the Killing cylinder around $\alpha_c$ is pointing downwards, as promised, and this finishes the proof on the case where $A$ is not diagonal and either $a^2+bc \leq 0$ or $b \neq 0$. Now, we treat the simpler case of $A$ being given by

\begin{equation}\label{Asecondcase}
A = \Md{1+a}{0}{c}{1-a}.
\end{equation}

\noindent On this case, we have that $Q_{22}(z) = e^{2(a-1)z}$ and $G_1(z) = (3+a) e^{2(a-1)z}$ (as, from \eqref{dmai0} we obtain $a_{11}(z) = e^{(1+a)z}$ and $a_{12}(z) = 0$). Thus, the mean curvature of a $\pi-$graph to a function $u(x,y) = f(x)$ is given by

\begin{equation*}
H = \left.\left.\frac{e^{2(a+1)f}}{2W^3}\right[f^{\prime\prime} + (3+a)\left(f^{\prime}\right)^2 + 2 e^{-2(1+a)f}\right],
\end{equation*}

\noindent and we can proceed the proof as done on the previous case.\end{proof}

\subsection{Existence of Scherk-like graphs: Proof of Theorem \ref{thmScherks}}\label{secProofScherk}

On this section we prove Theorem \ref{thmScherks}. The proof is an standard argument of convergence, with the difference that we are going to look at the graphs sometimes vertically (as $\pi-$graphs), to have geometrically defined barriers, and sometimes horizontally (as Killing graphs), so we can use techniques of Killing graphs and elliptic partial differential equations. 

\begin{figure}
\centering
\begin{subfigure}[b]{0.48\textwidth}
\includegraphics[width=\textwidth]{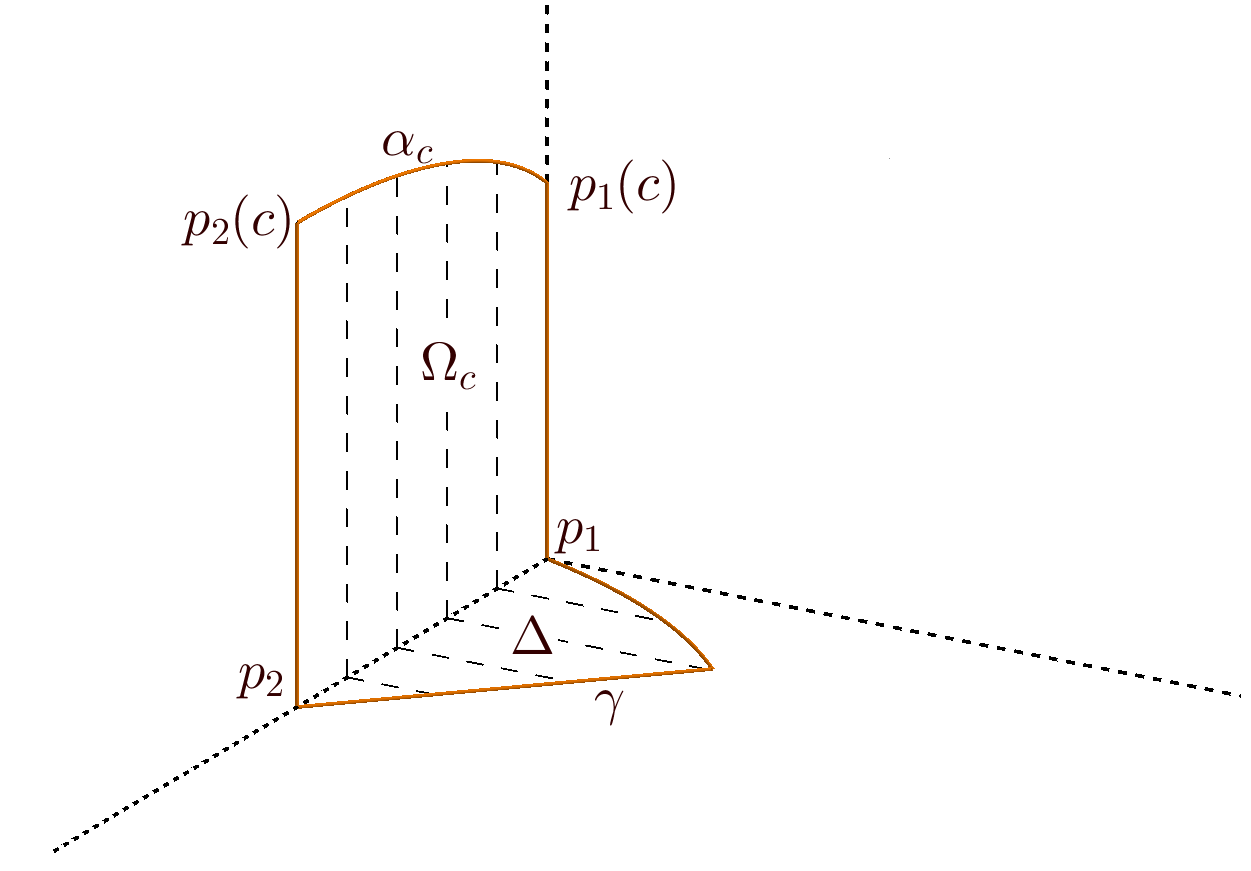}
\end{subfigure}
\begin{subfigure}[b]{0.48\textwidth}
\includegraphics[width=\textwidth]{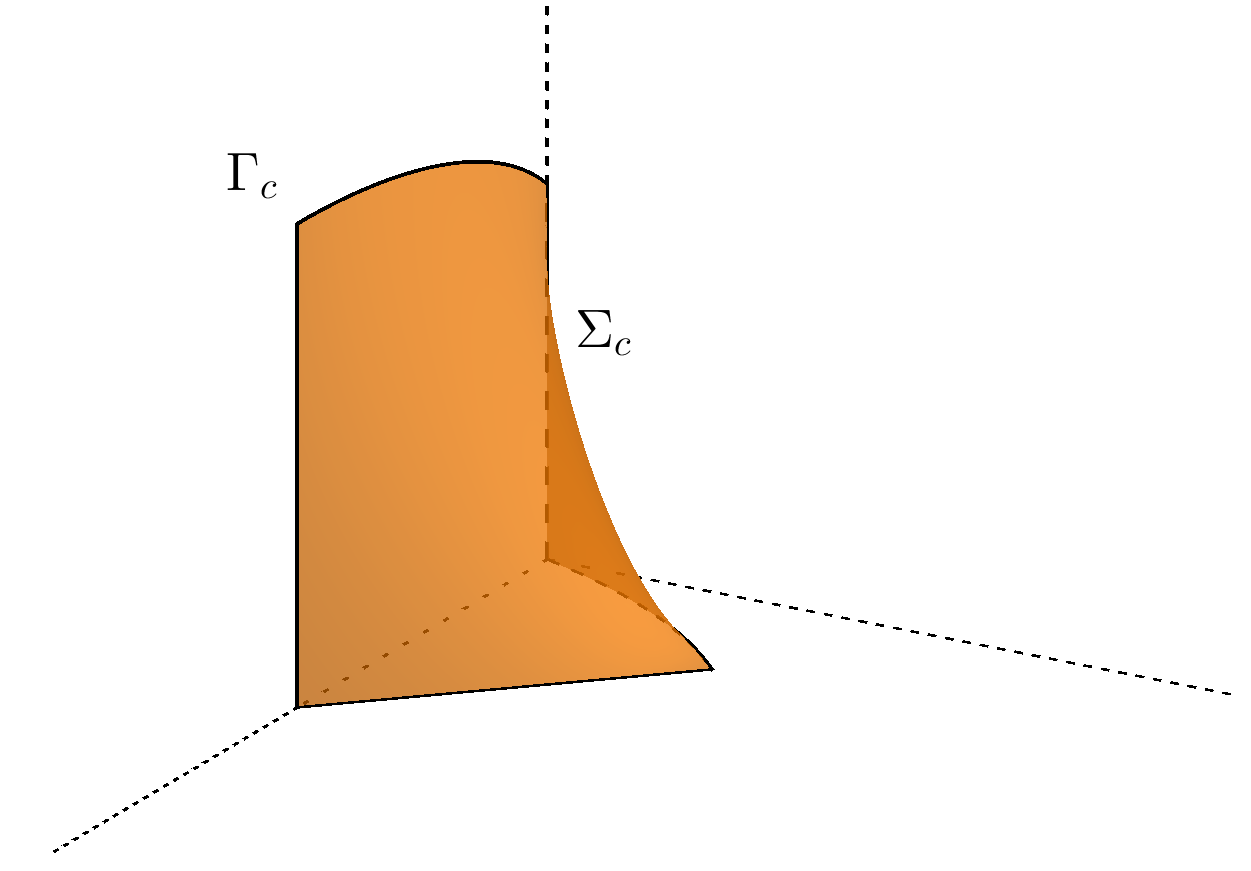}
\end{subfigure}
\caption[The curve $\Gamma_c$ and the surface $\Sigma_c$]{The surface $\Sigma_c$ (on the right) is both a $\pi-$graph over $\Delta$ and a $\partial_y-$Killing graph over $\Omega_c$, with $\partial \Sigma = \Gamma_c$ (the curve the left).\label{Sigmac}}
\end{figure}

\begin{proof}[Proof of Theorem \ref{thmScherks}]

Let $A \in M_2(\R)$ be any matrix with $\tr(A) \geq 0$ and let $L_0 > 0$ be the one given by Proposition \ref{propExhaus}. Let $p_1,\,p_2 \in \raz$ be such that $d(p_1,\,p_2) = L < L_0$ and without loss of generality assume $p_1 = (0,0,0)$ and $p_2 = (L,0,0)$. Let $\alpha = \{(x,0,0);0\leq x \leq L\}$ to be the segment joining $p_1 $ and $p_2$ and let $g:[0,\,L]\rightarrow \R$ be a convex, piecewise smooth function, with $g(0) = g(L) = 0$, meeting $\alpha$ on angles smaller than $\pi/2$ at $0$ and $L$, defining a curve $\gamma \subseteq \raz$,

\begin{equation*}
\gamma = \{(x,\,g(x),\,0)\in \raz ; \ 0\leq x \leq L\},
\end{equation*}

\noindent a curve smooth by parts with endpoints $p_1,\,p_2$ such that $\alpha \cup \gamma$ bounds a convex domain $\Delta \subseteq \raz$ (as on Figure \ref{Sigmac}, left).

We also let $\Omega_\infty = \alpha \times[0,\,+\infty)$ and, following the notation of Proposition~\ref{propExhaus}, let, for each $c\geq 0$,

\begin{equation*}
\Omega_c = \{(x,\,0,\,z)\in \rar; \ 0\leq x \leq L, \ 0\leq z \leq f_c(x)\},
\end{equation*}

\noindent and $\alpha_c = \{(x,\,0,\,f_c(x)); \ 0 \leq x \leq L\}$ to be the graph of $f_c$, in such a way that its $\partial_y-$Killing cylinder

\begin{equation*}
Cyl_{\partial_y}(\alpha_c) = \{(x,\,y,\,f_c(x)); \ 0 \leq x \leq L,\,y \in \R\}
\end{equation*}

\noindent has mean curvature vector pointing downwards. We also denote by

\begin{equation*}
p_1(c) = (0,0,f_c(0)),\ p_2(c) = (L,\,0,\, f_c(L))
\end{equation*}

\noindent the endpoints of $\alpha_c$, and we let, for $c \geq 0$, $\Gamma_c$ be a simple closed curve on $\rar$ given by (see Figure \ref{Sigmac}, left)

\begin{equation}
\Gamma_c = \gamma \cup \overline{p_1p_1(c)}\cup \alpha_c \cup \overline{p_2p_2(c)}.
\end{equation}

\begin{claim}\label{claimsigmac}
The curve $\Gamma_c$ as above bounds an \emph{unique} minimal $\pi$-graph $\Sigma_c$ over $\Delta$, which is also a $\partial_y-$Killing graph over $\Omega_c$.
\end{claim}

\noindent\emph{Proof of Claim \ref{claimsigmac}.} First, we notice that $\Gamma_c$ monotonically parametrizes $\partial \Delta$, then we can use Theorem 15.1 of \cite{MMPR3} to obtain a minimal, least area, $\pi-$graph $\Sigma_c$ with boundary $\partial \Sigma_c = \Gamma_c$. 

We are going to show that $\Sigma_c$ is a $\partial_y-$Killing graph, on the sense that there will exist a function $g_c:\Omegab_c\rightarrow \R$, smooth up to the boundary, such that $R(g_c) = 0$ (here $R$ will stand for the elliptic operator of minimal $\partial_y$-Killing graphs) and

\begin{equation}\label{SigmKil}
\Sigma_c = Gr_{\partial_y}(g_c) = \{(x,\,g_c(x,z),\,z);\ (x,0,z) \in \Omega_c\}.
\end{equation}

To begin with, as $\Sigma_c$ is a $\pi-$graph, there exists a function $u_c:\Delta \rightarrow \R$ such that

\begin{equation}\label{SigmaPi}
\Sigma_c = \graph(u) = \{(x,y,u_c(x,y)); \ (x,y,0) \in \Delta\}.
\end{equation}

\noindent We remark that $\Sigma_c$ is contained on the $\partial_y$-Killing cylinder over $\Omega_c$, so $0 \leq u_c(x,y) \leq f_c(x)$. Indeed, that $u > 0$ on the interior of $\Delta$ follows directly from the maximum principle. If there was an interior point $(x_0,\,y_0,\,0) \in \Delta$ such that $u_c(x_0,\,y_0) > f_c(x_0)$, then we could consider the family $Cyl_{\partial_y}(\alpha_t)$, for $t >c$, and obtain a last contact point, interior for both $\Sigma_c$ and $Cyl_{\partial_y}(\alpha_t)$, so the mean curvature of $Cyl_{\partial_y}(\alpha_t)$ would point upwards, in contradiction with Proposition~\ref{propExhaus}. 

Let $q = (x,0,z) \in \Omega_c$ be an interior point and consider $\mathcal{O}(q)$ to be the orbit of $q$ with respect to the flux $\varphi_t$ of the Killing field $\partial_y$, which is the horizontal line $\mathcal{O}(x,0,z) = \{(x,y,z);y \in \R\}$. Then we notice that $\mathcal{O}(x,0,z) \cap \Sigma_c$ is never empty, otherwise $\Sigma_c$ would not be simply connected and then it could not be a $\pi-$graph over $\Delta$.

Moreover, the intersection $\mathcal{O}(x,0,z) \cap \Sigma_c$ must be always a single point: if it contained two (or more) points $q_i = \varphi_{t_i}(q)\in \Sigma_c$, with $0<t_1< t_2$, then for $t_0 = t_2 - t_1 > 0$, $\varphi_{t_0}(\Sigma_c) \cap \Sigma_c \neq \emptyset$. Now, as $\varphi_t(\partial\Sigma_c) \cap \Sigma_c = \emptyset$ for all $t\neq 0$ by construction, we can consider the last contact point between $\varphi_t(\Sigma_c)\cap \Sigma_c$, and it will be interior for both $\Sigma_c$ and $\varphi_t(\Sigma_c)$, a contradiction with the maximum principle.

We denote $(x,g_c(x,z),z) = \Sigma_c \cap \mathcal{O}(x,0,z)$. This implies that $\Sigma_c$ can be written as \eqref{SigmKil}, but we still do not have the regularity on $g_c$. In order to prove that $g_c$ is smooth, we begin by proving that $\grad(g_c)$ is never vertical. 

Let $q \in \Omega_c$ be any interior point and consider a small ball $B = B_{\Omega_c}(q,r) \subseteq \text{int}(\Omega_c)$ such that $Cyl_{\partial_y}(\partial B)$ has mean curvature vector pointing inwards. Consider the following problem on $B$:

\begin{equation}\label{KillingProb}
\left\{ \begin{array}{l}
R(w) = 0, \text{ on } \text{int}(B)\\
w\vert_{\partial B} = g_c\vert_{\partial B},
\end{array}\right.
\end{equation}

\noindent where $R$ is the mean curvature operator for $\partial_y-$Killing graphs. In other words, we are looking for a minimal $\partial_y-$Killing graph over a small ball on $\Omega_c$ that coincides with $\Sigma_c$ on its boundary. 

If $\Phi = g_c\vert_{\partial B}$ was of class $C^{2,\,\alpha}$, we could simply use the existence result due to M. Dajczer and J. H. de Lira, Theorem 1 of \cite{DL}\footnote{We note that the hypothesis on the Ricci curvature on \cite{DL} is used uniquely to obtain an a priori estimate for the height of the graph, which is satisfied on our setting by the maximum principle.} to obtain a solution to \eqref{KillingProb}. However, at this point we can only guarantee that $\Phi$ is of class $C^0$, so we need to use an approximation argument. Let $(\Phi^{\pm}_n)_{n \in \mathds{N}} \subseteq C^{2,\,\alpha}(\partial B)$ be two sequences of $C^{2,\,\alpha}$ functions, converging to $\Phi$ and such that 

\begin{equation}\label{monotone}
\Phi^-_{n} \leq \Phi^-_{n+1} \leq \Phi \leq \Phi_{n+1}^+ \leq \Phi_{n}^+,
\end{equation}

\noindent for every $n \in \mathds{N}$. By Theorem 1 of \cite{DL}, there are functions $w^{\pm}_n \in C^{2,\,\alpha}(\overline{B})$ with minimal $\partial_y-$Killing graph and such that $w_n^{\pm}\vert_{\partial B} = \Phi^{\pm}_n$. From \eqref{monotone} we obtain that the sequences $w^{\pm}_n$ also are monotone, $w^-_n$ is non-decreasing and $w^+_n$ is non-increasing, both uniformly bounded. To obtain the convergence of the sequences $w^{\pm}_n$ to a solution of \eqref{KillingProb}, we are going to use some recent gradient estimates for Killing graphs obtained by J.-B. Casteras and J. Ripoll on \cite{CR}:

\begin{theorem}[Theorem 4, \cite{CR}]\label{gradestCR} Let $M$ be a Riemannian manifold and let $Y$ be a Killing field. Let $\Omega$ be a Killing domain in $M$ and let $o \in \Omega$ and $r > 0$ such that the open geodesic ball $B_\Omega(o,r)$ is contained in $\Omega$. Let $u \in C^3(B_\Omega(o,r))$ be a negative function whose $Y-$Killing graph has constant mean curvature $H$. Then there is a constant $L$ depending only on $u(o)$, $r$, $\abs{Y}$ and $H$ such that $\norma{\grad(u)(o)}\leq L$.
\end{theorem}

All functions $w^{\pm}_n$ have uniform bounds on the $C^0$ norm, thus Theorem \ref{gradestCR} above implies that there are uniform gradient estimates on compact subsets of $B$. This implies that both sequences will converge on the $C^2-$norm to a function $w \in C^{2}(B)\cap C^{0}(\overline{B})$, which is a solution of \eqref{KillingProb}. Now, just use the flux of $\partial_y$ and the same translation argument as before to obtain that $w$ coincides with $g_c$ on $B$, in particular the gradient of $g_c$ is not vertical on interior points of $\Omega_c$, as claimed.

Now we use the relation $(x,g_c(x,z),z) = (x,y,u_c(x,y))$ to prove that $g_c$ is actually smooth \emph{up to the boundary}, with the unique exceptions of $p_1,\,p_2$, $p_1(c),\,p_2(c)$, where $\partial \Omega_c$ is not smooth, and the finite number of points where $g$ is not differentiable. Just notice that $u_c$ is smooth up to the boundary (except on the points where $\partial \Delta$ is not differentiable) and that the gradient of $u_c$ is never horizontal on $\partial \Delta$, by the boundary maximum principle. Moreover, it follows from last argument that $\grad(u_c)$ is never horizontal on interior points of $\Delta$, so $g_c$ is also smooth up to the boundary, with the exceptions above.

Finally, we remark that this argument proves that \emph{any} minimal $\pi-$graph $S$ with $\partial S = \Gamma_c$ is a Killing graph, then we obtain that $\Sigma_c$ is unique, which proves Claim~\ref{claimsigmac}. $\hfill \diamondsuit$

\vspace{0.1cm}

Now, we notice that the uniqueness of $\Sigma_c$, given by Claim~\ref{claimsigmac}, implies that the correspondence $c \mapsto g_c$ is continuous, and, as defined, the functions $g_c$ satisfied, on the boundary of $\Omega_c$:
\begin{equation*}
g_c(0,z) = g_c(L,z) = 0, \ g_c\vert_{\alpha_c} = 0, \ g_c(x,0) = g(x).
\end{equation*}

Again, as $\Sigma_c$ is a $\pi-$graph over $\Delta$, it is contained on the $\pi-$cylinder over $\Delta$, and this can be translated to the horizontal setting as the inequality

\begin{equation}\label{gradestdepois}
0 \leq g_c(x,z) \leq g(x),
\end{equation}

\noindent for every $(x,\,0,\,z) \in \Omega_c$. Moreover, the usual argument using the translations given by the flux of $\partial_y$ shows that the sequence $g_c$ is monotonically increasing on $c$, that is $g_c(x_0,\,z_0) \leq g_{c^\prime}(x_0,\,z_0)$, for every $(x_0,\,0,\,z_0)\in \Omega_c$ and for every $c^\prime \geq c$. In particular the sequence will converge (as it is bounded) pointwise for some function $g_\infty : \Omega_\infty \rightarrow \R$, such that $0 \leq g_\infty \leq g$. Next claim will show that the convergence is actually on the $C^{2}-$norm, so $Gr_{\partial_y}(g_\infty) = \Sigma_\infty$ is a minimal surface of $\rar$.

\begin{claim}\label{claimconvergence} When $c\rightarrow \infty$, the functions $g_c$ converge on the $C^{2}-$norm to $g_\infty:\Omega_\infty \rightarrow \R$.
\end{claim}

\noindent \emph{Proof of Claim \ref{claimconvergence}.} We are going to use the same argument of Claim \ref{claimsigmac}, via gradient and height estimates for Killing graphs. Let $K \subseteq \Omega_\infty$ be a compact set contained on $\overline{\Omega_\infty}$ with $C^{2,\,\alpha}$ boundary, as on the below picture. As it holds $g_c(x,z) \leq g(x)$, follows that the height of $g_c$ is uniformly bounded on $K$, so we can use Theorem~\ref{gradestCR} to obtain an uniform bound for the norm of the gradient of every $g_c$ on interior points of $K$.

\begin{figure}[H]
\centering
\includegraphics[width=0.7\textwidth]{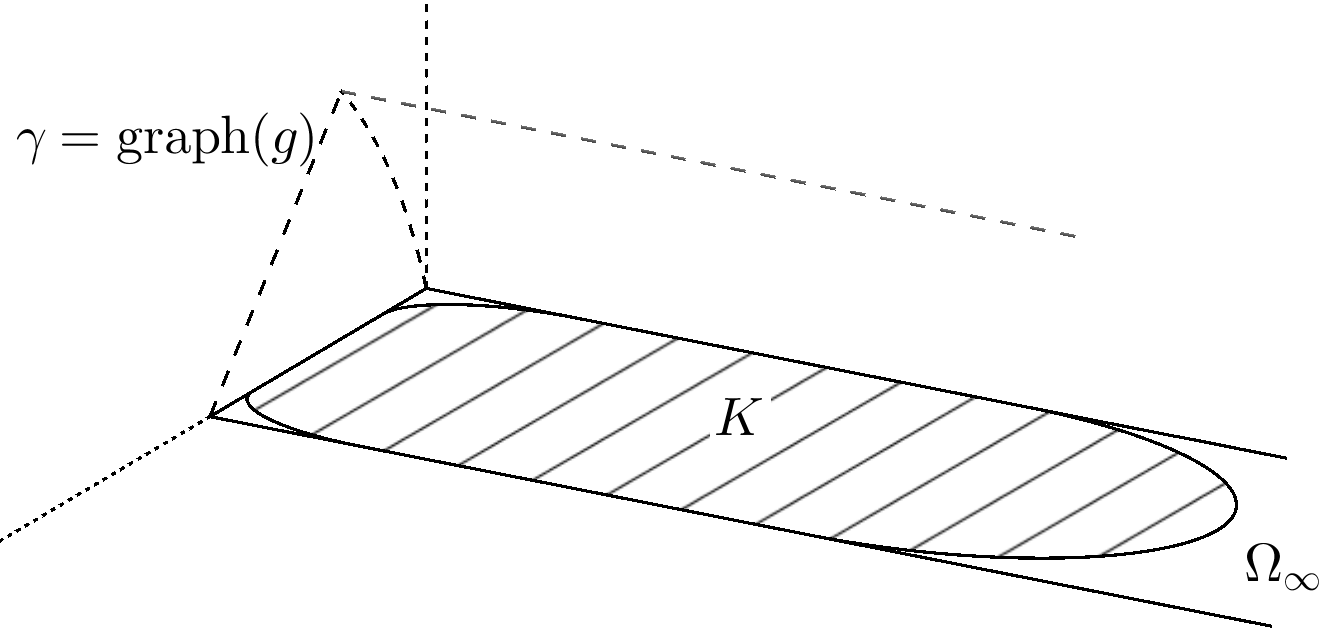}
\end{figure}

Now, we remark that \eqref{gradestdepois}, together with the assumption that the angle $\gamma$ makes with $\alpha$ at $p_1$ and $p_2$ is less than $\pi/2$, implies that every $g_c$ satisfy a uniform gradient estimate also along the boundary of $K$, as $g(0) = g(L) = 0$. As $g_c\vert_K \in C^{2,\,\alpha}(K)$ is smooth up to the boundary, this implies an uniform (not depending on $c$) estimate for the gradient of $g_c$ on $K$.

Now, we just take an exhaustion of $\Omega_\infty$ by compact sets and, by using an standard argument via the theory of partial elliptic equations, we obtain that a subsequence of the $g_c$ converges to $g_\infty$ on the $C^{2}-$norm. In particular, as the sequence is monotone and converges pointwise, follows that the convergence is smooth on the whole $\Omega_\infty$.$\hfill\diamondsuit$

\vspace{0.1cm}

From this claim we obtain that $\Sigma_\infty$ is a minimal surface of $\rar$, and that its boundary is 

\begin{equation*}
\partial \Sigma_\infty = \Gamma_\infty = \gamma \cup \big(\{p_1\}\times [0,\,\infty)\big) \cup \big(\{p_2\}\times [0,\,\infty)\big),
\end{equation*}

\noindent as, on the convergence of $g_c$ we had $g_c = 0$ on the horizontal segments of $\partial \Omega_c$ and $g_c = g$ on $\alpha$, with $g_c$ smooth up to the boundary.

Now, in order to finish the proof of the theorem, it remains to show that $\Sigma_\infty$ is nowhere vertical and that it is unique. The uniqueness comes directly from the fact that it was built as a Killing graph, and that every other surface with such boundary is contained on the $\partial_y$-Killing cylinder over $\Omega_\infty$.

To show that $\Sigma_\infty$ is nowhere vertical, we go back to analyse the problem using $\pi-$graphs. First, if there was an \emph{interior} point $p \in \Sigma_\infty$ such that $T_p\Sigma_\infty$ was vertical, we observe that $\Sigma_\infty  $ and $T_p\Sigma_\infty$ would be two minimal surfaces of $\rar$  tangent to each other at $p$. Then, there are at least two curves, meeting transversely at $p$ on the intersection $T_p\Sigma_\infty\cap\Sigma_\infty$, so $\Sigma_\infty$ cannot be a $\pi-$graph on a neighbourhood of $p$, so it is a $\pi-$cylinder over some line segment\footnote{If $\beta \subseteq \raz$ is a smooth curve, the $\pi-$cylinder $\beta \times [0,\,\infty)$ is minimal if and only if $\beta$ is a line segment: to see this, just use the foliation of $\rar$ by vertical planes which are parallel to the vertical plane generated by the endpoints of $\beta$. It also follows from the more general formula $H(x,y,z) = k_g(x,y)e^{-z\tr(A)}$, where $k_g(x,y)$ denotes the geodesic curvature of $\beta$ on the point $(x,y,0)$. The proof of this formula is a simple computation.} $\beta$ contained on $\partial \Delta$. Also if the point $p \in \partial \Sigma_\infty$ was a boundary point, then the boundary maximum principle would give the same conclusion. Next claim is to show that $\Sigma_\infty$ meets $\pi^{-1}(\gamma)$ uniquely on $\gamma$, so $\Sigma_\infty \supseteq \big( \beta \times[0,\,\infty)\big)$ is a contradiction.

\begin{claim}\label{claimLast} $\Sigma_\infty \cap \pi^{-1}(\gamma) = \gamma$.
\end{claim}

\noindent \emph{Proof of Claim \ref{claimLast}.} We use the same barrier technique that A. Menezes, \cite{Men}: Let $\gamma_i$ be a smooth component of $\gamma$ and let $p \in \gamma_i$ be any point. Consider $L$ the vertical plane of $\rar$ containing the tangent line to $\gamma_i$ at $p$ (this is well defined even for $p \in \partial \gamma_i$, as $\gamma_i$ is smooth). As $\gamma$ is convex, this implies that $\Delta$ is contained on the same connected component of $\rar$ defined by $L$, so also does $\Sigma_\infty$.

Recall the functions $u_c:\Delta \rightarrow\R$ defined via \eqref{SigmaPi}. Let $c_0 = \sup_{\Delta} u_0$ and let $c_2 > c_1 > c_0$ be any two large numbers. We consider a rectangle $R \subseteq L$ with boundary $\partial R = r_1\cup r_2 \cup s_1 \cup s_2$ given by two parallel horizontal segments $r_1$ and $r_2$ and two vertical segments $s_1$ and $s_2$, such that $s_1 \subseteq \{z = c_1\}$ and $s_2 \subseteq \{z = c_2\}$ (see Figure \ref{figbarrier}). Moreover, we assume that $R$ projects into $\raz$ in a compact segment $r \ni p$ with endpoints $q_1 = \pi(s_1)$ and $q_2 = \pi(s_2)$, contained on the same half-space determined by $\{y=0\}$ (the vertical plane containing $\alpha$) and with $q_2$ outside of $\Delta$ ($q_1 \in \Delta$ can happen if and only if $\gamma_i$ is a line segment).

\begin{figure}
\centering
\includegraphics[width=0.8\textwidth]{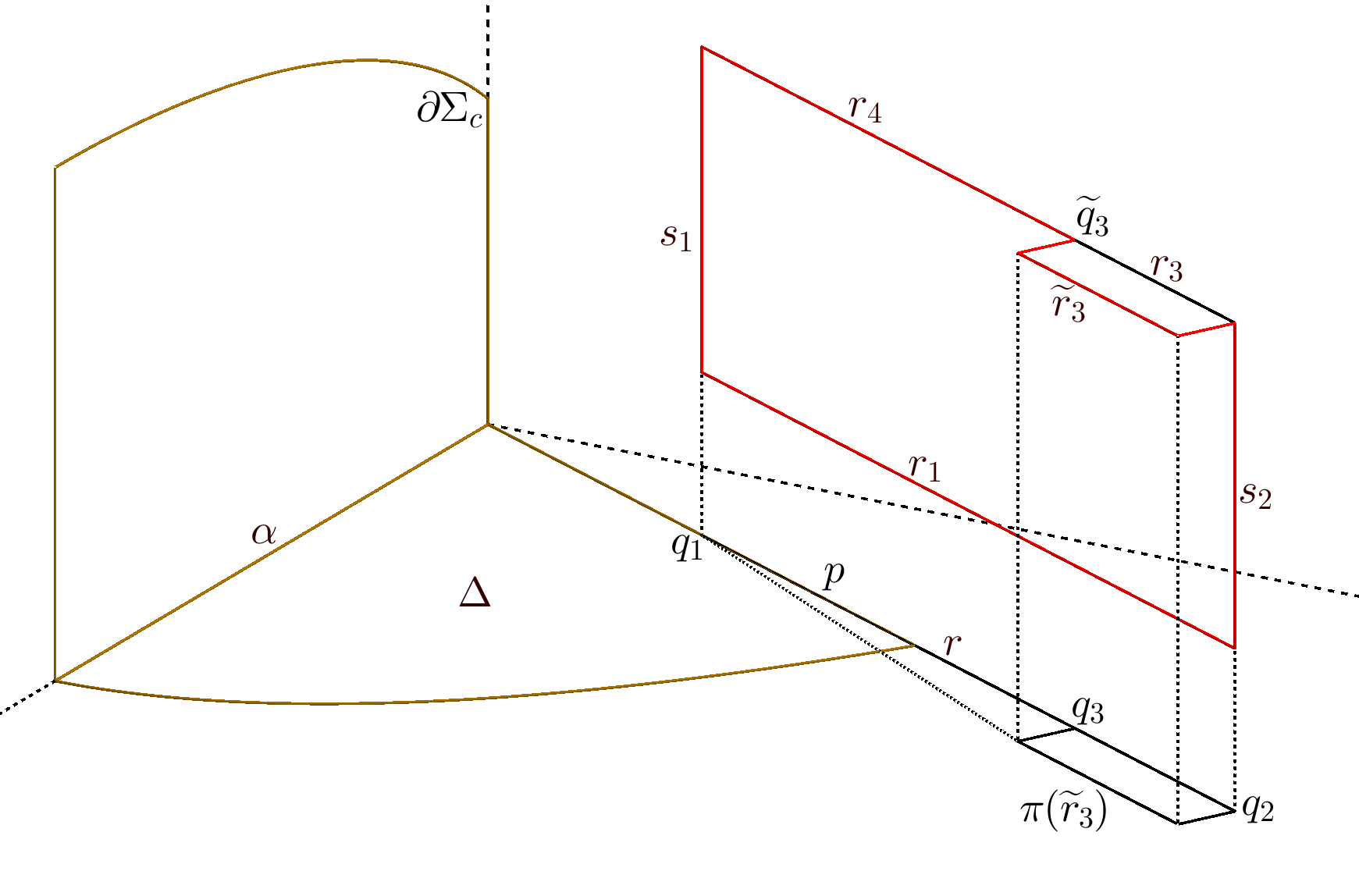}
\caption[The construction of the barrier $\widetilde{R}$]{The construction of the barrier $\widetilde{R}$, by deforming $\partial R$ over $r_3$.\label{figbarrier}}
\end{figure}

Let $q_3 \in \pi(R)$ be a point interior to the projection of $R$ but that lies outside of $\Delta$. Then, $\widetilde{q}_3 = \pi^{-1}(q_3)\cap r_2$ divides $r_2$ in two compact segments $r_3 \cup r_4$, $r_3$ projecting entirely outside of $\Delta$ and with $p \in \pi(r_4)$.

We remark that $L$ is stable, as it is transversal to a (horizontal) Killing field, and in particular, it follows from the useful criteria due to D. Fischer-Colbrie and R. Schoen, Theorem 1 of \cite{FS} (also proved on Proposition 1.32 of the book by T. Colding and W. Minicozzi, \cite{CM}) that $R$ is strictly stable, thus small perturbations of $\partial R$ give rise to minimal surfaces with the perturbed boundary.

We change $r_2$ by making a parallel translation of $r_3$ on the direction of the half-space that contains $\Sigma_\infty$, whose projection still does not intersect $\Delta$, joined by two small segments and denote such curve $\widetilde{r}_3$, in such a way that $r_3 \cup \widetilde{r}_3$ bounds a small rectangle on the horizontal plane $\{z = c_2\}$. We assume this perturbation is small in such way that its projection does not intersect $\Delta$. Let $\widetilde{R}$ be a minimal surface of $\rar$ whose boundary is the perturbed rectangle $r_1\cup \widetilde{r}_3 \cup r_4\cup s_1 \cup s_2 $. Such surface is nowhere vertical and its contained on the convex hull of its boundary, in particular it is contained on $\{z \geq c_1\}$ and on the same half space that $\Sigma_\infty$ with respect to the plane $L$.

Now, it is easy to see that $\pi(\widetilde{R}) \cap \Delta \neq \emptyset$, as otherwise $\widetilde{R}\cap R$ would have a interior contact point. Moreover, $\widetilde{R}$ is above $u_0$ on $\pi(\widetilde{R}) \cap \Delta$, by the construction of $\widetilde{R}$. Then, if $\Sigma_\infty \cap \pi^{-1}(\gamma_i) \neq \gamma_i$, we would have that $\Sigma_\infty \cap \widetilde{R} \neq \emptyset$, thus, for some $\ell > 0$ there would be a first contact point between $\Sigma_\ell$ and $\widetilde{R}$. As $\partial \Sigma_\ell$ does not intersect the convex hull of $\partial \widetilde{R}$, it does not intersect $\widetilde{R}$. Moreover, neither $\partial \widetilde{R}$ can intersect $\Sigma_\ell$, as this would imply such point would be on the plane $L$, so $\Sigma_\ell$ would have a vertical tangent plane. Then this contact point is going to be interior for both, reaching to a contradiction that proves the claim.  $\hfill\diamondsuit$

From this claim and from the argument previously done, we obtain that $\Sigma_\infty$ is actually more than a Killing graph, it is also a $\pi-$graph, nowhere vertical, which finishes the proof of the theorem.\end{proof}

\noindent\hspace{0.6cm}\'Alvaro Kr\"uger Ramos

\noindent\hspace{0.6cm}Departamento de Matematica

\noindent\hspace{0.6cm}Univ. Federal do Rio Grande do Sul

\noindent\hspace{0.6cm}Av. Bento Gon\c{c}alves 9500

\noindent\hspace{0.6cm}91501-970 -- Porto Alegre -- RS -- Brazil

\noindent\hspace{0.6cm}alvaro.ramos@ufrgs.br

\noindent\hspace{0.6cm}{\small {(supported by CNPq - Brasil)} }

\end{document}